\documentclass[12pt]{amsart}
\usepackage{amssymb,latexsym}

\title{Discrete Morse Theory for totally non-negative flag varieties}

\author{Konstanze Rietsch}%
\author{Lauren Williams}
\address{Department of Mathematics,
            King's College London,
            Strand, London
            WC2R 2LS}%
\address{Department of Mathematics,
            Harvard University,
            Cambridge, MA 02138}
\email{konstanze.rietsch@kcl.ac.uk}%
\email{lauren@math.harvard.edu}
\thanks{
The first author is funded by
EPSRC grant EP/D071305/1.  The second author is partially
supported by the NSF}%



\usepackage{pstricks, pst-node, amssymb}

\psset{unit=1pt, arrowsize=4pt, linewidth=.7pt}
\psset{linecolor=blue}
\newgray{grayish}{.90}
\newrgbcolor{embgreen}{0 .5 0}

\def\vblack(#1, #2)#3{\cnode*[linecolor=black](#1, #2){3}{#3}}
\def\vwhite(#1,#2)#3{\cnode[linecolor=black,fillcolor=white,fillstyle=solid](#1,#2){3}{#3}}
\countdef\x=23
\countdef\y=24
\countdef\z=25
\countdef\t=26

\def\tbox(#1,#2)#3{
\x=#1 \y=#2
\multiply\x by 12
\multiply\y by 12
\z=\x \t=\y
\advance\z by 12
\advance\t by 12
\psline(\x,\y)(\x,\t)(\z,\t)(\z,\y)(\x,\y)
\advance\x by 6
\advance\y by 6
\rput(\x,\y){{\bf #3}}}

\font\co=lcircle10

\def\jr{\rotatedown{\smash{\raise2pt\hbox{\co \rlap{\rlap{\char'005} \char'007}}
               \raise6pt\hbox{\rlap{\vrule height6.5pt}}
                \raise2pt\hbox{\rlap{\hskip4pt \vrule
          height0.4pt depth0pt
                width7.7pt}}}}}
\def\textcross{\ \smash{\lower4pt\hbox{\rlap{\hskip4.15pt\vrule height14pt}}
                \raise2.8pt\hbox{\rlap{\hskip-3pt \vrule height.4pt depth0pt
                        width14.7pt}}}\hskip12.7pt}

\def\textelbow{\ \hskip.1pt\smash{\raise2.75pt%
                \hbox{\co \hskip 4.15pt\rlap{\rlap{\char'004} \char'006}
                \lower6.8pt\rlap{\vrule height3.5pt}
                \raise3.6pt\rlap{\vrule height3.5pt}}
                \raise2.8pt\hbox{%
                  \rlap{\hskip-7.15pt \vrule height.4pt depth0pt
width3.5pt}%
                  \rlap{\hskip4.05pt \vrule height.4pt depth0pt
width3.5pt}}}
                \hskip8.7pt}

\usepackage{amsmath, amsthm, amssymb, amsbsy}
\usepackage{amsfonts, latexsym, stmaryrd, amscd, xy}
\usepackage[mathscr]{eucal}
\usepackage{epsfig}

\newtheorem{theorem}{Theorem}[section]

\newtheorem{proposition}[theorem]{Proposition}
\newtheorem{lemma}[theorem]{Lemma}

\newtheorem{corollary}[theorem]{Corollary}

\newtheorem{remark}[theorem]{Remark}

\newtheorem{definition}[theorem]{Definition}

\numberwithin{equation}{section}

\usepackage{amsfonts}
\usepackage{xy}
\usepackage{amssymb}
\usepackage{amsmath}
\def\interior{\mathrm{Int}}
\def\Carrier{\mathrm{Carrier}}
\newcommand{\Z}{\mathbb Z}
\newcommand{\To}{\longrightarrow}
\newcommand{\R}{\mathbb R}

\newcommand{\N}{\mathbb N}
\newcommand{\C}{\mathbb C}
\newcommand{\PPP}{\mathbb{P}}
\newcommand{\PP}{\mathcal{P}}

\newcommand{\K}{\mathcal{K}}
\newcommand{\F}{\mathcal{F}}

\newcommand{\B}{\mathcal{B}}
\newcommand{\I}{\mathcal{I}}
\newcommand{\E}{\mathcal{E}}

\newcommand{\Q}{\mathcal{Q}}

\DeclareMathOperator{\bd}{bd}

\DeclareMathOperator{\rank}{rank}

\DeclareMathOperator{\Hom}{Hom}

\DeclareMathOperator{\Sl}{SL}

\DeclareMathOperator{\Last}{Last}

\newcommand{\thmrefer}[1]{\renewcommand\thetheorem
  {\protect\ref{#1}}\addtocounter{theorem}{-1}}

\xyoption{all}
\CompileMatrices

\begin{document}

\begin{abstract}
In a seminal 1994 paper \cite{Lusztig3}, 
Lusztig extended the theory of total positivity by introducing
the totally non-negative part $(G/P)_{\geq 0}$ of an arbitrary 
(generalized, partial) flag variety
$G/P$.  He referred to this space as a ``remarkable polyhedral subspace,"
and conjectured a decomposition into cells,
which was subsequently 
proven by the first author \cite{Rietsch1}.  In \cite{Wil} the second
author made the concrete conjecture that this cell decomposed 
space is the next best thing to a polyhedron, by conjecturing 
it to be a {\it regular} 
CW complex that is homeomorphic to a closed ball.
 In this article we use discrete Morse theory to 
prove this conjecture up to homotopy-equivalence. 
Explicitly, we prove that the  
boundaries of the cells are homotopic to spheres, and
the closures of cells are 
contractible.  
The latter part generalizes a result of Lusztig's \cite{Lusztig1}, 
that $(G/P)_{\geq 0}$ -- the closure
of the top-dimensional cell -- is contractible.
Concerning our result on the boundaries of cells, even the special case
that the boundary of the top-dimensional cell 
$(G/P)_{> 0}$ is homotopic to a sphere, is new for all
$G/P$ other than projective space.  
\end{abstract}

\maketitle

\setcounter{tocdepth}{1}
\tableofcontents

\section{Introduction}
The classical theory of total positivity studies matrices whose
minors are all positive.  Lusztig dramatically 
generalized this theory with 
a 1994 paper \cite{Lusztig3} in which he introduced
the totally positive part of a reductive group $G$
(totally positive matrices are recovered when
$G$ is a general linear group). Lusztig also defined 
the (totally) positive 
and (totally) non-negative parts
$(G/P)_{>0}$ and $(G/P)_{\geq 0}$  
of an arbitrary (generalized, partial) flag variety
$G/P$.
Lusztig referred to $(G/P)_{\geq 0}$ as a ``remarkable polyhedral 
subspace"
\cite{Lusztig3}, 
and conjectured a decomposition into cells,
which was subsequently 
validated by the first author \cite{Rietsch1}.  This cell
decomposition has a unique top-dimensional cell, the totally
positive part $(G/P)_{> 0}$;  the totally non-negative part
$(G/P)_{\geq 0}$ is the closure of this cell.

Lusztig \cite{Lusztig1} has proved that the totally nonnegative part of the
(full) flag variety is contractible, which implies the same result
for any partial flag variety.  
More generally, in 1996 Lusztig asked whether
the closure of each cell
of $(G/P)_{\geq 0}$ is contractible \cite{LusztigQuestion}, but 
this problem has remained
open until now.  By analogy with toric varieties,
one might wonder whether even more is true --
whether  
$(G/P)_{\geq 0}$ 
is 
homeomorphic to a ball, and in that case, 
whether there is a homeomorphism 
to a polyhedron mapping cells
to faces:
indeed, there is a notion of total positivity for toric varieties,
and the non-negative part of a toric variety is homeomorphic -- via the 
moment map -- to its moment polytope \cite{Fulton}.
It turns out that $(G/P)_{\geq 0}$ cannot be modeled by a polyhedron in the above 
sense: for example, the totally non-negative part of
the Grassmannian $Gr_{2,4}(\R)$ has one top-dimensional
cell of dimension $4$ and four $3$-dimensional cells, but
there is no $4$-dimensional polytope with four facets.
Nevertheless, in \cite{Wil} the second author conjectured that 
 $(G/P)_{\geq 0}$ together with its cell decomposition is the next best thing to a polyhedron,
that is, it is a {\it regular} CW complex -- the closure of each cell 
is homeomorphic to a closed ball and the boundary of each cell is homeomorphic to 
a sphere.

The goal of this paper is to apply combinatorial and 
topological methods 
in order to address this conjecture.
Indeed, the past thirty years
have seen a wealth of literature
designed to
facilitate the interplay between combinatorics and geometry
(see \cite{DK}, \cite{B1}, \cite{BW1}, \cite{BW2}, \cite{B2}).
In particular,
in a 1984 paper \cite{B2}, Bjorner recognized that regular CW complexes
are combinatorial objects in the following sense:
if $Q$ is the poset of closed cells in a regular CW decomposition of
a space $X$, then the order complex (or nerve) $\Vert Q \Vert$
is homeomorphic to $X$.  Furthermore, he gave criteria \cite{B2}
for recognizing
when a poset is the face poset of a regular CW complex: for example,
if a poset is {\it thin} and {\it shellable} then it is the face
poset of some regular CW complex homeomorphic to a ball.

In \cite{Rietsch2}, the first author described the poset $Q$ of closed 
cells of $(G/P)_{\geq 0}$, and in \cite{Wil}, 
the second author applied techniques from poset topology 
to the poset of closed cells of $(G/P)_{\geq 0}$.  In particular,
she showed that the poset is thin and shellable.
It follows that the order complex $\Vert Q \Vert$ is homeomorphic 
to a ball, and by Bjorner's
results, $Q$ is the poset of cells of a regular CW decomposition of a ball.
These results were the motivation for her conjecture
that the cell decomposition of $(G/P)_{\geq 0}$ is a regular
CW decomposition of a ball.  

While the statement that $Q$ is the poset of cells of a regular 
CW decomposition of a ball is an extremely strong combinatorial result, 
one cannot use it to deduce any corresponding topological consequences 
for the original space $(G/P)_{\geq 0}$.  Even to deduce results about the
Euler characteristics of closures of cells requires further topological
information about how cells are glued together, i.e. knowing that the 
cell complex is a {\it CW complex}. This was proved some ten years after 
the discovery of the cell decomposition, in \cite{PSW, RW}. 

To obtain new topological information 
about the  CW complex $(G/P)_{\geq 0}$, we turn in this paper to 
another technique, namely  Forman's discrete Morse theory \cite{RF}.
The main theorem of discrete Morse theory is set up to
provide a sequence of collapses for cells in a CW complex, which
preserves the homotopy-type of the CW complex.  To use
it, one must input some combinatorial data --  a {\it discrete Morse function}, which
specifies the sequence of collapses -- and  check a number of topological 
hypotheses.  Most notably, one must make sure that whenever one cell $C_1$
is collapsed into a cell $C_2$ whose closure contains $C_1$, $C_1$ is a {\it regular face}
of $C_2$.

In this paper we use a blend of combinatorial and topological arguments
to apply discrete Morse theory to $(G/P)_{\geq 0}$.
Our main result is the following.
\begin{theorem}\label{main}
Let $(G/P)_{\geq 0}$ be an arbitary (generalized, partial) flag variety.  
The closure of each cell of $(G/P)_{\geq 0}$ is collapsible, hence
contractible.  Furthermore, the boundary of each cell is homotopy-equivalent
to a sphere.  In particular, 
$(G/P)_{\geq 0}$ is contractible and its boundary
is homotopy-equivalent to a sphere.
\end{theorem}

While it was known already that $(G/P)_{\geq 0}$ is contractible by work of Lusztig,
this theorem also identifies the homotopy type of its boundary, and of the 
closures of the smaller cells and their boundaries. Namely, we prove the 
conjecture that $(G/P)_{\geq 0}$ is a regular CW 
decomposition of a ball up to 
homotopy-equivalence.

We note that much of the technical difficulty of proving our
main results stems from the fact that the attaching maps
that we constructed for cells in \cite{RW} are defined 
in a non-explicit way in terms of Lusztig's canonical basis.
Identifying enough pairs of cells $(C_1,C_2)$ with
$C_1$ a provably regular face of $C_2$, and then demonstrating
regularity, requires an intricate analysis of parameterizations of 
cells and of what happens when parameters go to infinity.
Our arguments rely in a fundamental way 
on positivity properties of the canonical basis.

The combinatorial component of our arguments is also nontrivial.
For every cell $C$ in $(G/P)_{\geq 0}$, we find a {\it Morse matching}
on the poset of cells in the closure of $C$, with a unique critical 
cell of dimension $0$, such that matched pairs of cells are regular.
This requires us to identify appropriate Morse matchings of intervals
in Bruhat order; the matchings we construct generalize certain
{\it special matchings} found by Brenti \cite{Brenti} in the context of 
Kazhdan-Lusztig theory.
An essential tool in our proofs is Dyer's notion of reflection orders
and his EL-labeling
of Bruhat order \cite{Dyer}.
Along the way, we give a link between poset topology and discrete Morse
theory, building on work of Chari to provide an 
algorithm for passing explicitly from an EL-labeling of a CW poset
to a Morse matching.

For the time being, there is no simple strategy 
for proving that closures of cells are homeomorphic to balls.
One might hope to use recent work of  Hersh \cite{Hersh}
on determining when an attaching map 
for a CW complex is a homeomorphism on its entire domain.
However, there is only one known CW structure for $(G/P)_{\geq 0}$
(the one we gave in \cite{RW}), and its attaching maps are not 
homeomorphisms.

It is worth noting that to our knowledge this paper represents
one of the first instances of the application of 
combinatorial tools (poset topology and discrete Morse theory)
to a topological space which is not a simplicial or regular cell complex,
and which arose outside the context of combinatorics.
Indeed, most 
of the tools of poset topology are designed to analyze the 
order complex of a poset (a simplicial complex), e.g.
the order complex of Bruhat order, 
the partition lattice, the lattice of subgroups of a finite group
\cite{Wachs}.
Similarly, discrete Morse theory is most readily applied to
simplicial complexes and regular CW complexes
(as opposed to general CW complexes), because in
these situations one does not have to check extra topological
hypotheses before collapsing cells.
In light of this, it is not surprising that virtually all of 
the many applications of discrete Morse theory to date have been to
simplicial complexes,
e.g.\ complexes of 
$t$-colorable graphs, complexes of connected and biconnected graphs,
complexes of not $i$-connected graphs; 
see \cite{Forman2} for an interesting survey.  

Therefore we hope that this paper will be valuable not only
in shedding light on the topology of $(G/P)_{\geq 0}$, but also 
in demonstrating the applicability of combinatorial tools 
to topological spaces outside the world of combinatorics.

\textsc{Acknowledgements:}
L.W.: It is a pleasure to thank Anders Bjorner
for stimulating conversations about poset topology 
and discrete Morse theory
at the Mittag-Leffler Institute in May 2005.
I would also like to thank
Sergey Fomin, 
Arun Ram, and David Speyer.
Finally, I am grateful to Tricia Hersh for useful feedback
and for drawing my attention to \cite{Kozlov}.

\section{Preliminaries on algebraic groups and flag varieties} 

We start with some preliminaries.

\subsection{Pinnings.}\label{Pin}
Let $G$ be a semisimple, simply connected linear algebraic group over $\C$ split over
$\R$, with split torus $T$.  We identify $G$ (and related spaces)
with their real points and consider them with their real topology.
Let $X(T) = \Hom(T, \R^*)$ and $\Phi \subset X(T)$ the set of roots.
Choose a system of positive roots $\Phi^+$.  We denote by $B^+$ the
Borel subgroup corresponding to $\Phi^+$ and by $U^+$ its unipotent
radical.  We also have the opposite Borel subgroup $B^-$ such that
$B^+ \cap B^- = T$, and its unipotent radical $U^-$.

Denote the set of simple roots by
$\Pi = \{\alpha_i \ \vline \ i \in I \} \subset \Phi^+$.
For each $\alpha_i \in \Pi$ there is an associated homomorphism
$\phi_i : \Sl_2 \to G$.
Consider the $1$-parameter subgroups in $G$ (landing in $U^+, U^-$,
and $T$, respectively) defined by
\begin{equation*}
x_i(m) = \phi_i \left(
                   \begin{array}{cc}
                     1 & m \\ 0 & 1\\
                   \end{array} \right) ,\
y_i(m) = \phi_i \left(
                   \begin{array}{cc}
                     1 & 0 \\ m & 1\\
                   \end{array} \right) ,\
\alpha_i^{\vee}(t) = \phi_i \left(
                   \begin{array}{cc}
                     t & 0 \\ 0 & t^{-1}\\
                   \end{array} \right) ,
\end{equation*}
where $m \in \R, t \in \R^*, i \in I$.
The datum $(T, B^+, B^-, x_i, y_i; i \in I)$ for $G$ is
called a {\it pinning}.  The standard pinning for
$\Sl_d$ consists of the diagonal, upper-triangular, and lower-triangular
matrices, along with the simple root subgroups
$x_i (m) = I_d + mE_{i,i+1}$ and
$y_i (m) = I_d + mE_{i+1,i}$ where
$I_d$ is the identity matrix and $E_{i,j}$ has a $1$ in
position $(i,j)$ and zeroes elsewhere.

\subsection{Folding}\label{s:automorphism} If $G$ is not simply laced,
 then one can
    construct a simply laced group $\dot G$ and an automorphism
     $\tau$ of $\dot G$ defined over $\R$, such that there is
  an isomorphism, also defined over $\R$, between $G$ and
    the fixed point subset $\dot G^\tau$ of $\dot G$. Moreover
     the groups $G$ and $\dot G$ have compatible pinnings.  Explicitly
     we have the following.
	 
   Let $\dot G$ be simply connected and simply laced.
    We apply the same notations as in Section \ref{Pin} for $G$, but with a dot,
		     to our simply laced group $\dot G$. So we have a pinning
			  $(\dot T,\dot B^+,\dot B^-, \dot x_i,\dot y_i, i\in \dot I)$ of $\dot G$,
		   and $\dot I$ may be identified with the vertex set of the Dynkin
	    diagram of $\dot G$.
	    
	  Let $\sigma$ be a permutation of $\dot I$ preserving
	   connected components of the Dynkin diagram, such that,
	    if $j$ and $j'$ lie in  the same orbit under $\sigma$
	     then they are {\it not} connected by an edge.
		  Then $\sigma$ determines an automorphism $\tau$ of $\dot G$
		   such that
    \begin{enumerate}
	     \item
		  $\tau(\dot T)=\dot T$,
		   \item
		    $\tau(x_i(m))=x_{\sigma(i)}(m)$ and $\tau(y_i(m))=y_{\sigma(i)}(m)$ for all
	     $i\in \dot I$ and $m\in \R$.
	  \end{enumerate}
			   In particular $\tau$ also preserves $\dot B^+,\dot B^-$. Let $\bar I$ denote
	    the set of $\sigma$-orbits in $\dot I$, and for $\bar i\in \bar I$, let
	     \begin{eqnarray*}
		  x_{\bar i}(m)&:=&\prod_{i\in \bar i}\ x_i(m),\\
	   y_{\bar i}(m)&:=&\prod_{i\in \bar i}\ y_i(m).
    \end{eqnarray*}
The fixed
  point group $\dot G^{\tau}$ is a simply
  laced, simply connected algebraic group with  pinning
  $({\dot T}^\tau,\dot B^{+ \tau},\dot B^{-\, \tau},  x_{\bar i}, y_{\bar i},      \bar i\in \bar I)$.
  There exists, and we choose, $\dot G$ and $\tau$ such that  $\dot G^{\tau}$  
  is isomorphic to our group $G$ via an isomorphism compatible 
with the pinnings.

\subsection{Flag varieties}
The Weyl group $W = N_G(T) / T$ acts on $X(T)$ permuting the roots
$\Phi$.  
We set $s_i:= \dot{s_i} T$ where
$\dot{s_i} :=
                 \phi_i \left(
                   \begin{array}{cc}
                     0 & -1 \\ 1 & 0\\
                   \end{array} \right).$
Then any $w \in W$ can be expressed as a product
$w = s_{i_1} s_{i_2} \dots s_{i_m}$ with $\ell(w)$ factors.  
This gives $W$ the structure of a Coxeter group; we will 
assume some 
basic knowledge 
of Coxeter systems and Bruhat order as in 
\cite{Humphreys}. 
We set
$\dot{w} = \dot{s}_{i_1} \dot{s}_{i_2} \dots \dot{s}_{i_m}$.
It is known that $\dot w$ is independent of the reduced expression 
chosen.

We can identify the flag variety $G/B$ with the variety 
$\B$ of Borel subgroups, via
\begin{equation*}
gB^+ \Longleftrightarrow g \cdot B^+ := gB^+ g^{-1}.
\end{equation*}
We have the Bruhat decompositions
\begin{equation*}
\B = \sqcup_{w \in W} B^+ \dot{w} \cdot B^+ =
    \sqcup_{w \in W} B^- \dot{w} \cdot B^+
\end{equation*}
of $\B$ into $B^+$-orbits called {\it Bruhat cells},
and $B^-$-orbits called {\it opposite Bruhat cells}.

\begin{definition}
For $v, w \in W$ define
\begin{equation*}
\mathcal R_{v,w}: = B^+ \dot{w} \cdot B^+ \cap B^- \dot{v} \cdot B^+.
\end{equation*}
\end{definition}
The intersection $\mathcal R_{v,w}$ is non-empty precisely if $v \leq w$
in the Bruhat order,
and in that case is irreducible of dimension $\ell(w) - \ell(v)$,
see \cite{KL}.

Let $J \subset I$.  The parabolic subgroup
$W_J \subset W$
corresponds to a parabolic subgroup $P_J$ in $G$
containing $B^+$.  Namely,
$P_J = \sqcup_{w \in W_J} B^+ \dot{w} B^+$.
Consider the variety $\PP^J$ of all parabolic subgroups of $G$
conjugate to $P_J$.  This variety can be identified with the partial
flag variety $G/P_J$ via
\begin{equation*}
gP_J \Longleftrightarrow gP_J g^{-1}.
\end{equation*}
We have the usual projection from the full flag variety to
a partial flag variety which takes the form
$\pi = \pi^J: \B \to \PP^J$, where
$\pi(B)$ is the unique parabolic subgroup of type $J$ containing
$B$.

\section{Total positivity for flag varieties}\label{Sec:TotPos}

\subsection{The totally non-negative part of $G/P$ and its cell decomposition}

\begin{definition} \cite{Lusztig3}
The totally non-negative
part $U_{\geq 0}^-$ of $U^-$ is defined to be the semigroup in
$U^-$ generated by the $y_i(t)$ for $t \in \R_{\geq 0}$.

The
totally non-negative part of $\B$ (denoted by $\B_{\geq 0}$ or
by $(G/B)_{\geq 0}$) is defined by
\begin{equation*}
\B_{\geq 0} := \overline{ \{u \cdot B^+ \ \vline \ u \in U_{\geq 0}^- \} },
\end{equation*}
where the closure is taken inside $\B$ in its real topology.

The totally non-negative
part of a partial flag variety
$\PP^J$ (denoted by $\PP^J_{\geq 0}$ or by $(G/P_J)_{\geq 0}$)
 is defined to be
$\pi^J (\B_{\geq 0})$.
\end{definition}

Lusztig \cite{Lusztig3, Lusztig2} introduced natural decompositions of $\B_{\geq 0}$ and
$\PP^J_{\geq 0}$.

\begin{definition}  \cite{Lusztig3}
For $v, w \in W$ with $v \leq w$, let
\begin{equation*}
\mathcal R_{v,w ; >0} := \mathcal R_{v,w} \cap \B_{\geq 0}.
\end{equation*}
\end{definition}

We write $W^J$ (respectively $W^J_{max}$) for the set of minimal
(respectively maximal) length coset representatives of $W/W_J$.

\begin{definition} \cite{Lusztig2} \label{index}
Let $\I^J \subset W^J_{max} \times W_J \times W^J$ be the set of
triples $(x,u,w)$ with the property that $x \leq wu$.
Given $(x,u,w) \in \I^J$,
we define $\mathcal P_{x,u,w; >0}^J := \pi^J(\mathcal R_{x,wu; >0}) = \pi^J(\mathcal R_{xu^{-1},w; >0}).$
\end{definition}

The first author \cite{Rietsch1} proved that
$\mathcal R_{v,w; >0}$ and
$\mathcal P_{x,u,w; >0}^J$ are semi-algebraic cells of dimension
$\ell(w)-\ell(v)$ and $\ell(wu) - \ell(x)$, respectively.

\subsection{Parameterizations of cells}\label{s:param}

In \cite{MarRie:ansatz}, Marsh and the first author gave parameterizations
of the cells $\mathcal R_{v,w; >0}$,  which we now explain.

Let $v\le w$ and let $\mathbf w=(i_1,\dotsc, i_m)$ encode a
reduced expression $s_{i_1}\dotsc s_{i_m}$ for $w$. Then there
exists a unique subexpression $s_{i_{j_1}}\dotsc s_{i_{j_k}}$
for $v$ in $\mathbf w$ with the
property that, for $l=1,\dotsc,k$,
\begin{equation*}
s_{i_{j_1}}\dotsc s_{i_{j_{l}}}s_{i_r}>
s_{i_{j_1}}\dotsc s_{i_{j_{l}}}  \quad \text{whenever $j_l< r\le j_{l+1}$,}
\end{equation*}
where $j_{k+1}:=m$. This is the ``rightmost reduced subexpression" for $v$ in
$\mathbf w$, and is called the  ``positive  
subexpression" in \cite{MarRie:ansatz}. It was originally
introduced by Deodhar~\cite{Deodhar}.  We
use the notation
\begin{eqnarray*}
\mathbf v_+&:=&\{j_1,\dotsc, j_k\},\\
\mathbf v_+^c&:=&\{1,\dotsc, m\}\setminus \{j_1,j_2,\dotsc, j_k\},
\end{eqnarray*}
for this special subexpression for $v$ in $\mathbf w$. Note that this notation 
only makes sense in the context of a fixed $\mathbf w$.

Now we can define the map
\begin{eqnarray*}
\phi_{\mathbf v_+,\mathbf w}:(\C^*)^{\mathbf v^c_+} &\to& \mathcal R_{v,w},\\
        ( t_r)_{r\in\mathbf v^c_+}&\mapsto &g_1\dotsc g_m\cdot B^+,
\end{eqnarray*}
where
\begin{equation*}
g_r=\begin{cases}
\dot s_{i_r}, &\text{ if $r\in\mathbf v_+$,}\\
y_{i_r}(t_r) & \text{ if $r\in \mathbf v^c_+$.}
\end{cases}
\end{equation*}

\begin{theorem}\cite[Theorem~11.3]{MarRie:ansatz}\label{t:param}
The restriction of $\phi_{\mathbf v_+,\mathbf w}$ to $(\R_{>0})^{\mathbf v_+^c}$
defines an isomorphism of semi-algebraic sets,
$$
\phi^{>0}_{\mathbf v_+,\mathbf w}: (\R_{>0})^{\mathbf v_+^c} \to \mathcal R_{v,w;>0}.
$$
\end{theorem}

We note that this parameterization generalizes Lusztig's parametrization of 
totally nonnegative cells in $U^-_{\ge 0}$ from \cite{Lusztig3}. Namely 
  $
U^-_{\ge 0}=\bigsqcup_{w\in W} U^-_{>0}(w),
  $
for
$$
   U^-_{>0}(w):=\{y_{i_1}(t_1)y_{i_2}(t_2)\dots y_{i_m}(t_m) \ | \ t_i\in \R_{>0}\},
$$
where $\mathbf w=(i_1,\dotsc, i_m)$ is a/any reduced expression of $w$. Clearly,
$
\mathcal R^{>0}_{1,w}=U^-_{>0}(w)\cdot B^+.
$  

\subsection{Change of coordinates under braid relations}\label{s:coord} 
In the simply laced case
there is a simple change of coordinates \cite{Lusztig3, Rie:MSgen}
which describes how
two parameterizations of the same cell are related when considering
two reduced expressions which differ by a commuting relation or
a braid relation.

If $s_i s_j = s_j s_i$ then $y_i(a) y_j(b)=y_j(b) y_i(a)$
   and $y_i(a) \dot s_j = \dot s_j y_i(a)$.

If $s_i s_j s_i = s_j s_i s_j$ then
\begin{enumerate}
\item
$y_i(a) y_j(b) y_i(c) = y_j\left(\frac{bc}{a+c}\right) y_i(a+c) y_j\left(\frac{ab}{a+c}\right)$,
\item $y_i(a) \dot s_j y_i(b) = y_j\left(\frac ba\right) y_i(a) \dot s_j x_j\left(\frac ba\right)$,
\item
    $\dot s_j \dot s_i y_j(a)=
    y_i(a) \dot s_j \dot s_i$.
\end{enumerate}
In case (2) Lemma 11.4 from \cite{MarRie:ansatz} implies that the factor 
$x_j(\frac ba)$ disappears into $B^+$ without affecting the remaining
parameters when this braid relation is applied in the parametrization
of a totally nonnegative cell.

The changes of coordinates have also
been computed for more general braid relations
and have been observed to be invertible, subtraction-free, homogeneous
rational transformations \cite{BZ,Rie:MSgen}.

\subsection{Total positivity and canonical bases for simply laced $G$}\label{s:CanonBasis}

\bigskip

Assume that $G$ is simply laced.  Let $\bf U$ be the enveloping
algebra of the Lie algebra of $G$; this can be defined by generators
$e_i, h_i, f_i$ ($i\in I$) and the Serre relations.  For any
dominant weight
$\lambda\in X(T)$ embedded into $\mathfrak h^*$, 
there is a finite-dimensional simple $\bf U$-module $V(\lambda)$
with a non-zero vector $\eta$ such that $e_i\cdot \eta = 0$
and $h_i \cdot\eta = \lambda(h_i) \eta$ for all $i \in I$.   
The
pair $(V(\lambda), \eta)$ is determined up to unique isomorphism.

There is a unique $G$-module structure on $V(\lambda)$ such that
for any $i\in I, a\in \R$ we have
\begin{equation*}
x_i(a) = \exp(a e_i):V(\lambda) \to V(\lambda), \qquad
y_i(a) = \exp(a f_i): V(\lambda) \to V(\lambda).
\end{equation*}
Then
$x_i(a)\cdot\eta = \eta$ for all $i\in I$, $a\in \R$, and
$t\cdot \eta = \lambda(t) \eta$ for all $t \in T$.  Let $\B(\lambda)$
be the canonical basis of $V(\lambda)$ that contains $\eta$
\cite{Lus:CanonBasis}.
We now collect some useful facts about the canonical basis.

\begin{lemma}\label{lem:positive2} \cite[1.7(a)]{Lusztig2}.
For any $w\in W$, the vector $\dot w\cdot \eta$ is the unique element of
$\B(\lambda)$ which lies in the extremal weight space
$V(\lambda)^{w(\lambda)}$.  In particular, $\dot w\cdot \eta \in \B(\lambda)$.
\end{lemma}

We define $f_i^{(p)}$ to be $\frac{f_i^p}{p!}$.  

\begin{lemma}\label{lem:canonbasis2}
Let $s_{i_1} \dots s_{i_n}$ be a reduced expression for $w\in W$.
Then there exists $a \in \N$ such that 
$f_{i_1}^{(a)} \dot s_{i_2} \dot s_{i_3} \dots \dot s_{i_n} \cdot \eta = 
\dot s_{i_1} \dot s_{i_2} \dots \dot s_{i_n}\cdot \eta$. Moreover,
$f_{i_1}^{(a+1)} \dot s_{i_2} \dot s_{i_3} \dots \dot s_{i_n}\cdot \eta = 0$.
\end{lemma}

\begin{proof}
This follows from Lemma \ref{lem:positive2} and properties of the 
canonical basis, see e.g.\ the proof of \cite[Proposition 28.1.4]{Lus:Quantum}.
\end{proof}

\section{$(G/P)_{\geq 0}$ as a CW complex: attaching maps using toric varieties} \label{Sec:Attach}

Recall that a {\it CW complex} is a union $X$  of cells with additional
requirements on how cells are {\it glued}: in particular, for each cell $\sigma$,
one must define a (continuous) {\it attaching map} $h: B \to X$ where $B$ is a closed ball,
such that the restriction of $h$ to the interior of $B$ is a homeomorphism with 
image $\sigma$.

Even given the parameterizations of cells, it is not obvious how to define
attaching maps.  One needs to extend the domain of each map 
$\phi^{>0}_{\mathbf v_+, \mathbf w}$ from 
 $(\R_{>0})^{\mathbf v_+^c}$ (an open ball) to a closed ball.  However,
a priori it is not clear how to let the parameters approach $0$ and infinity.
In this section we explain how, 
following earlier work for Grassmannians \cite{PSW}, the authors \cite{RW}
defined attaching maps for the cells and proved that the 
cell decomposition of $(G/P)_{\geq 0}$ is a CW complex.

Lemma \ref{important} below is the key to defining attaching maps.  It says that 
one can compactify $(\R_{>0})^{\mathbf v_+^c}$ inside a 
toric variety related to the parameterization, obtaining a closed ball
(the non-negative part of the toric variety).
We refer to \cite{Fulton} and \cite{Sottile} for the basics on toric varieties
and their non-negative parts.  Let $S\in \Z^r$  be a finite set whose elements are ordered,
$\mathbf {m_0},\dotsc, \mathbf {m_K},$  and
thought of as corresponding to monomials, 
$\mathbf t^{\mathbf {m_j}}=t_1^{m_{j,1}}t_2^{m_{j,2}}\dotsc t_r^{m_{j,r}}$. We let 
$X_S$ denote the toric subvariety
of $\mathbb P ^K$ associated to $S$ as in \cite{Sottile}, and 
 $X_S^{>0}$ and $X_S^{\ge 0}$ its positive and 
non-negative parts, respectively. Explicitly,
$X_S$ is the closure of the image of the associated map 
\begin{equation}\label{e:chi}
\chi=\chi_S:\ \mathbf t=(t_1,\dotsc, t_r)\mapsto [\mathbf t^{\mathbf m_0}, \mathbf t^{\mathbf m_1},
\dotsc, \mathbf t^{\mathbf m_K}]
\end{equation}
from $(\C^*)^r$ to $\mathbb P^K$, while $X_S^{>0}$, and $X_S^{\ge 0}$ are obtained
as the image of $\mathbb R_{>0}^r$ and its closure.

A fact which is 
crucial here is that
$X_S^{\ge 0}$ is homeomorphic to a closed ball. More specifically,
$X_S$ has a moment map which gives a homeomorphism 
from $X_S^{\ge 0}$  to the convex hull $B_S$ of $S$.

\begin{lemma}\label{important} \cite{PSW}
Suppose we have a map $\phi: (\R_{>0})^r \to \mathbb P^{N}$ given by
\begin{equation*}
(t_1, \dots , t_r) \mapsto [p_1(t_1,\dots,t_r), \dots , p_{N+1}(t_1,\dots,t_r)],
\end{equation*}
where the $p_i$'s are Laurent polynomials with positive coefficients.  Let $S$ be the
set of all exponent vectors in $\Z^r$ which occur among the (Laurent) monomials
of the $p_i$'s, and let $B_S$ be the convex hull of the points of $S$.
Then the map $\phi$ factors through the totally positive part
$X_S^{>0}$ of the toric variety, giving a map
$\Phi_{>0}: X_S^{>0} \to \PPP^{N}$.  Moreover $\Phi_{>0}$ extends continuously to the
closure to give a well-defined map
$\Phi_{\ge 0}:X_S^{\ge 0} \to \overline{\Phi_{>0}(X_{S}^{>0})}$.
Note that if we precompose with the isomorphism $B_S\cong X_S^{\ge 0} $given by
the moment map, we can consider the domain of $\Phi_{\ge 0}$ to be 
the polytope $B_S$, a closed ball.
\end{lemma}

The following result constructs attaching maps for cells of
$(G/P)_{\geq 0}$ \cite{RW}.

\begin{theorem} \cite{RW} \label{RWTheorem}
For any $G/B$ we can construct a positivity preserving embedding
$i:G/B \to \mathbb P^N$, for some $N$, with the following property.
For any totally non-negative cell $\mathcal R_{x,w; >0}$ and parameterization
$\phi^{>0}_{\mathbf{x_+}, \mathbf{w}}$ as in Section~\ref{s:param}, the composition
$$
i\circ\phi^{>0}_{\mathbf x_+, \mathbf{w}}:(\R_{>0})^{\mathbf x^c_+}\overset{\sim}\To\mathcal 
R_{x,w; >0}\hookrightarrow
\mathbb P^N
$$
takes the form
$$
i \circ\phi^{>0}_{\mathbf x_+, \mathbf{w}} : \mathbf t=  ( t_r)_{r\in\mathbf x^c_+}\mapsto [p_1(\mathbf t),\dotsc,p_{N+1}(\mathbf t)],
$$
where the $p_j$'s are polynomials with positive coefficients.
Applying Lemma \ref{important} to $i \circ\phi^{>0}_{\mathbf x_+, \mathbf{w}}$,
we get an attaching map 
$\Phi^{\geq 0}_{\mathbf x_+, \mathbf{w}}: X^{\geq 0}_{\mathbf x_+, \mathbf w} \to
\overline{\mathcal R_{x,w;>0}}$ where the non-negative toric variety 
$X^{\geq 0}_{\mathbf x_+,\mathbf w}$ 
is homeomorphic to its moment polytope $B_{\mathbf x_+,\mathbf w}$.

\end{theorem}

In Theorem \ref{RWTheorem}, the map $i$ is defined as follows.
When $G$
is simply laced,
we consider 
the 
representation
$V=V(\rho)$ of $G$ with
a fixed highest weight vector $\eta$ 
and corresponding canonical basis $\mathcal B(\rho)$.
We let $i:\mathcal B\to \mathbb P(V)$ denote the embedding
which takes $g\cdot B_+\in \mathcal B$ to the line
$\left <g\cdot \eta\right >$.
This is the unique $g \cdot B_+$-stable line in $V$.  
We specify points in the projective space $\mathbb P(V)$
using homogeneous coordinates corresponding to
$\mathcal B(\rho)$. The theorem then follows using the 
positivity properties of the canonical basis in simply 
laced type.

If $G$ is not simply laced, we  use a folding argument 
to deduce the result from the simply laced case: 
the map $i$ is given by $i^{\prime}$ from 
Lemma \ref{commutative} as we will now explain.
Let $\dot G$ be the simply laced group with automorphism
corresponding to $G$.
We identify $G$ with $\dot G^\tau$ and use
  all of the notation from Section~\ref{s:automorphism}. For
  any $\bar i\in \bar I$
  there is a simple reflection $s_{\bar i}$ in $W$, which is
  represented in $\dot G$ by
  $$
  \dot s_{\bar i}:=\prod_{i\in\bar i}\dot s_i.
  $$
  In this way any reduced expression $\mathbf w=(\bar i_1,\bar i_2,\dotsc, \bar i_m)$ in $W$ gives rise to a reduced
  expression $\dot{\mathbf w}$ in $\dot W$ of length
  $\sum_{k=1}^m |\bar i_k|$, which is determined
  uniquely up to commuting elements \cite{Steinberg}.
  To a subexpression $\mathbf v$ of $\mathbf w$
  we can then associate a unique subexpression
  $\dot {\mathbf v}$ of $\dot{\mathbf w}$ in
  the obvious way.

  \begin{lemma}\cite{RW} \label{commutative}
Let $v,w$ be in $W$ with $v\le w$.
   \begin{enumerate}
   \item
   We have
   $$
  \mathcal R_{v,w;>0}=\dot{\mathcal R}_{v,w;>0}\cap \mathcal B^{\tau}.
   $$
   In particular the composition $i':\mathcal R_{v,w}\hookrightarrow \dot{\mathcal R}_{v,w}
   \to \mathbb P(V(\dot\rho))$ is positivity preserving.
   \item
   Suppose $\mathbf w=(\bar i_1,\dotsc, \bar i_m)$ is a reduced expression for $w$ in $W$, and $\mathbf v_+^c=(h_1,\dotsc,h_r)$ is the 
complement of the positive subexpression for $v$. Then we have a commutative diagram,
   \begin{equation*}
   \begin{CD}
   \mathcal R_{v,w;>0}& @>\iota>> &\dot{\mathcal R}_{v,w;>0}\\
   @A\phi^{>0}_{\mathbf v_+,\mathbf w}AA & & @AA\phi^{>0}_{\dot{\mathbf v}_+,\dot{\mathbf w}}A\\
   \R_{>0}^{\mathbf v_+^c} &@>\bar \iota>> & \R_{>0}^{\dot {\mathbf v}_+^c},
   \end{CD}
   \end{equation*}
   where the top arrow is the usual inclusion, the vertical arrows are both isomorphisms,
   and the map $\bar \iota$  has the form
   $$(t_1,\dotsc, t_r)\mapsto (t_1,\dotsc, t_1, t_2,\dotsc, t_2,\dotsc, t_r),$$
    where each $t_l$ is repeated $|\bar i_{h_l}|$ times on the right hand side.
   \end{enumerate}
   \end{lemma}

\begin{remark}
For partial flag varieties we can also use Theorem \ref{RWTheorem}
to construct an attaching map for each 
$\mathcal P^J_{x,u,w;>0}$.  The projection
$\pi^J:\mathcal R_{xu^{-1},w;>0}\to\mathcal P^J_{x,u,w;>0}$ is a 
homeomorphism so we take the composition 
$\Phi^{\geq 0}_{\mathbf{xu^{-1}_+},\mathbf w} \circ \pi^J$ as our attaching map.
\end{remark}

\begin{theorem} \cite{RW} \label{RWCW}
$(G/P)_{\geq 0}$ is a CW complex.
\end{theorem}

\section{The poset $\mathcal{Q}^J$ of cells of $(G/P_J)_{\geq 0}$ and a regularity criterion} 
\label{sec:regularity}

In this section we will  review the description of the face poset 
of $(G/P_J)_{\geq 0}$ which was 
given by 
the first author \cite{Rietsch2}.
We will then prove 
Theorem \ref{regularity}, giving a condition which ensures
that a cell $\sigma$ is a {\it regular face} of another cell $\tau$
with respect 
to the attaching map of $\tau$.

\begin{definition}\label{faceposet}
Let $K$ be a finite CW complex, and let $Q$ denote its set of cells.
The notation $\sigma^{(p)}$ indicates that $\sigma$ is a cell of dimension
$p$.  We write $\tau>\sigma$ if $\sigma \neq \tau$ 
and $\sigma \subset \overline{\tau}$, where
$\overline{\tau}$ is the closure of $\tau$, and 
we say $\sigma$ is a face of $\tau$.  
This gives $Q$ the structure of a partially ordered set,
which we refer to as the 
{\it face poset} of $K$.
Sometimes we will augment $Q$ by adding a least element $\hat{0}$:
in this case 
we will say that $\tau > \hat{0}$ for all $\tau$, and we will 
call this the {\it augmented face poset} of $K$.
\end{definition}

\begin{remark}\label{augmented}
Our notion of face poset agrees with the notion used by Forman \cite{RF}.
However, Bjorner \cite{B2} defines the face poset of a 
cell complex to be the poset of cells augmented by 
a least element $\hat{0}$ 
{\it and} a greatest element $\hat{1}$.   In this paper we will 
never add a $\hat{1}$ to a poset because all posets we consider
already have a unique greatest element.
\end{remark}

A description of the face poset of $(G/P_J)_{\geq 0}$ was given in 
\cite{Rietsch2}. See also the paper \cite{GoodearlYakimov}  of Goodearl and Yakimov, 
who independently defined an
isomorphic poset in their study of the $T$-orbits of symplectic leaves 
for a Poisson structure on $G/P_J$.

\begin{theorem}\cite{Rietsch2} \label{RTheorem}
Fix $W$ and $W_J$, the Weyl group and its parabolic subgroup
corresponding to $G/P_J$.  
Let $\Q^J$ denote the augmented face poset of 
$(G/P_J)_{\geq 0}$ with its decomposition into
totally nonnegative cells. The elements of $\Q^J$ are indexed 
by $\I^J\cup \hat{0}$, where $\I^J$ is as in Definition~\ref{index}.

The order relations in $\Q^J$ are described in terms of Weyl group 
combinatorics by
$$
\mathcal P_{x,u,w;>0}^J \leq {\mathcal P_{x',u',w';>0}^J}
$$
if and only if  
there exist $u_1, u_2 \in W_J$ with $u_1 u_2 = u$ and
$\ell(u) = \ell(u_1)+\ell(u_2)$, such that
$x'{u'}^{-1} \leq x u_2^{-1} \leq w u_1 \leq w'.$ Moreover $\hat{0}< P_{x,u,w;>0}^J $
for all $(x,u,w)\in \mathcal I^J$.
\end{theorem}

\begin{remark}
When $G/P_J$ is a (type A) Grassmannian,  $\Q^J$
is the poset of cells of the totally non-negative Grassmannian,
first studied by Postnikov \cite{Postnikov}.
\end{remark}

When $P_{x,u,w;>0}^J < P_{x',u',w';>0}^J$ and 
$\dim P_{x',u',w';>0}^J = \dim P_{x,u,w;>0}^J+1$,
we will write 
$P_{x,u,w;>0}^J \lessdot P_{x',u',w';>0}^J$.

Suppose a cell 
$\sigma^{(p)}$ is a face of $\tau^{(p+1)}$.  Let $B$ be a closed ball of dimension $p+1$,
and let $h: B \to K$ be the attaching map for $\tau$, i.e.\ $h$ is a continuous map that
maps $\interior(B)$ homeomorphically onto $\tau$.
The following definition is essential to discrete Morse theory for 
general CW complexes, as  collapses of cells 
must take place along {\it regular edges}.  

\begin{definition}\cite[Definition 1.1]{RF} \label{regularface}
We say that $\sigma^{(p)}$ is a regular face of $\tau^{(p+1)}$
(with respect to the attaching map $h$ for $\tau$)
and that $(\sigma, \tau)$ is a regular edge, if
\begin{enumerate}
  \item $h: h^{-1}(\sigma)\to \sigma$ is a homeomorphism,
  \item $\overline{h^{-1}(\sigma)}$ is a closed $p$-ball.
\end{enumerate}
\end{definition}

To use discrete Morse theory
in our situation we must find enough regular edges.
However, the toric varieties and 
attaching maps in Theorem \ref{RWTheorem} 
are constructed using the canonical
basis, and hence are not at all explicit.
Thus at first glance it might seem hopeless to deduce whether 
a cell $\sigma$ is a regular face of $\tau$ with respect 
to an attaching map $h$ for $\tau$.  
Fortunately, by the following result we do have a situation in which we can 
prove regularity of a pair of faces.
We will first prove Theorem \ref{regularity} in the case of complete
flag varieties, and then generalize it to partial flag varieties.

\begin{theorem}\label{regularity}
Consider 
$\mathcal P_{x,u,w; >0}^{J} \gtrdot \mathcal P_{x',u',w; >0}^J$ in $(G/P_J)_{\geq 0}$ and
let $\mathbf{w}=(i_1, \dotsc, i_m)$ 
be a reduced expression for $w$.
We call 
this pair of cells {\bf good} with respect to $\mathbf w$
if the positive  subexpression 
$\mathbf{x'u'^{-1}_+}$ is equal to
$\mathbf{xu^{-1}_+} \cup \{k\}$ and moreover 
$\mathbf{xu^{-1}_+}$ contains $\{k+1,\dotsc, m\}$.
In this case 
$\mathcal P_{x',u',w; >0}^J$ is a regular face of 
$\mathcal P_{x,u,w; >0}^J$ with respect to 
the attaching map $\Phi^{\geq 0}_{\mathbf{xu^{-1}_+},\mathbf w} \circ \pi^J$.
\end{theorem}

When the choice of reduced expression $\mathbf w$ and the attaching
map are clear from context, we will sometimes omit the phrase 
{\it with respect to $\mathbf w$} or {\it with respect to the 
attaching map}.

\begin{proposition}\label{almostregularface}
Choose a reduced expression 
$\mathbf w = (i_1, \dots ,i_m)$ for $w$, and  
suppose that the pair $\mathcal R_{v,w; >0} \gtrdot \mathcal R_{v',w; >0}$ 
is good with respect to $\mathbf w$.
Suppose that $\mathbf v_+$ and $\mathbf v'_+$ are related by 
$\mathbf v'_+=\mathbf v_+\cup\{k\}$.
Then $X_{\mathbf v'_+,\mathbf w}$ can be identified with a    
sub-toric variety of $X_{\mathbf v_+,\mathbf w}$, and its moment polytope 
$B_{\mathbf v'_+,\mathbf w}$ is
a facet of the moment polytope $B_{\mathbf v_+,\mathbf w}$
of $X_{\mathbf v_+,\mathbf w}$.
Moreover, the attaching map
$\Phi^{\geq 0}_{\mathbf v_+, \mathbf{w}}: X^{\geq 0}_{\mathbf v_+, \mathbf w} \to
\overline{\mathcal R_{v,w;>0}}$ restricts to 
$X^{\geq 0}_{\mathbf v'_+,\mathbf w}$ to give the attaching map 
$\Phi^{\geq 0}_{\mathbf v'_+,\mathbf w}$ for $\overline{\mathcal R_{v',w;>0}}$.  
\end{proposition} 

\begin{proof}
Let us first consider the case that $G$ is simply-laced.
By our assumptions the parameterizations of the two cells take the form
\begin{eqnarray*}
\phi^{>0}_{\mathbf v_+,\mathbf w} &=& g_1 \dots g_{k-1} y_{i_k}(t_k) \dot s_{i_{k+1}}
           \dots \dot s_{i_m} \cdot B^+,\\  
\phi^{>0}_{\mathbf v'_+,\mathbf w} &= &g_1 \dots g_{k-1} \dot s_{i_k} \dot s_{i_{k+1}}
           \dots \dot s_{i_m} \cdot B^+.
\end{eqnarray*} 
If we compose the parameterization $\phi^{>0}_{\mathbf v_+,\mathbf w}$ with the inclusion 
$i:\mathcal R_{v,w}\hookrightarrow\mathbb P^N$ from Theorem~\ref{RWTheorem}
we get a  map 
$$
\mathbf t=(t_{h_1},\dotsc, t_{h_r},t_k)\mapsto [p_1(\mathbf t),\dots,p_{N+1}(\mathbf t)],
$$ 
where 
the $p_j$'s are polynomials with positive coefficients.  We note that by the definition of the map $i$, 
which we recalled just after Theorem~4.2,
$$
 [p_1(\mathbf t),\dots,p_{N+1}(\mathbf t)]=
           \left<g_1\dotsc g_{k-1}y_{i_k}(t_k)\dot s_{i_{k+1}} \dots \dot s_{i_m} \cdot \eta\right>.
$$
Here 
we have identified $\mathbb P^N$ with $\mathbb P(V(\rho))$ using the canonical basis.

If we take the limit as $t_k \to \infty$ we obtain a new map
\begin{multline}\label{e:limitmap}
\mathbf t'=(t_{h_1},\dotsc, t_{h_r})\mapsto [p'_1(\mathbf t'),\dotsc, p'_{N+1}(\mathbf t')] \\
=\lim_{t_k\to\infty}\left<g_1\dotsc g_{k-1}y_{i_k}(t_k)\dot s_{i_{k+1}} \dots \dot s_{i_m} \cdot \eta\right>=
\left<g_1\dotsc g_{k-1}\dot s_{i_k} \dot s_{i_{k+1}} \dots \dot s_{i_m} \cdot \eta\right>.
\end{multline}
Here the last equality follows by applying $g_1\dotsc g_{k-1}$ to the identity
$$\lim_{t_k\to\infty}\left<y_{i_k}(t_k)\dot s_{i_{k+1}} \dots \dot s_{i_m} \cdot \eta\right>=
\left<\dot s_{i_k} \dot s_{i_{k+1}} \dots \dot s_{i_m} \cdot \eta\right>,
$$
which comes from expanding the action of $y_{i_k}(t_k)=\exp(t_k f_{i_k})$, using that
\begin{eqnarray*}
f_{i_k}^{(a)} \dot s_{i_{k+1}} \dots \dot s_{i_m} \cdot \eta &= &
\dot s_{i_k} \dot s_{i_{k+1}} \dots \dot s_{i_m} \cdot \eta,\\
f_{i_k}^{(a+1)} \dot s_{i_{k+1}} \dots \dot s_{i_m} \cdot \eta &= &0,
\end{eqnarray*}
for a positive integer $a$, by Lemma~\ref{lem:canonbasis2}.

By the same argument, $a$ as above is the highest power of $t_k$
appearing in any $p_j$. Therefore to obtain homogeneous coordinates
$p'_j(\mathbf t')$ for the limit point 
we may divide each $p_j(\mathbf t)$ by $t_k^a$ and 
 take
$$
p_j'(\mathbf t')=\lim_{t_k\to\infty}\ \frac 1{t_k^a}p_j(\mathbf t).
$$
The monomials of $p_j(\mathbf t)$ which don't vanish in this limit are precisely
those which are multiples of this maximal power, $t_k^a$. 

It follows that the toric variety $X_{\mathbf v'_+,\mathbf w}$ is the sub-toric
variety of $X_{\mathbf v_+,\mathbf w}$ which is given precisely by those 
monomials which are multiples of $t_k^a$ (and other coordinates set to zero). 
Its moment polytope can be identified
with the face of $B_{\mathbf v_+,\mathbf w}$ cut out by the hyperplane
$x_k=a$. Moreover from \eqref{e:limitmap} it follows that 
the attaching map $\Phi_{\mathbf v'_+,\mathbf w}^{\ge 0}:X^{\ge 0}_{\mathbf v'_+,\mathbf w}\to 
\overline{\mathcal R_{v',w;>0}}$ 
is the restriction of
$\Phi_{\mathbf v_+,\mathbf w}^{\ge 0}$. This also implies that 
$B_{\mathbf v'_+,\mathbf w}$, which is isomorphic to 
$X^{\ge 0}_{\mathbf v'_+,\mathbf w}$, has codimension $1$ in   
$B_{\mathbf v_+,\mathbf w}$, making it a facet.

In the non simply-laced case the proof is analogous. 
However now the attaching map for  $\overline{\mathcal R_{v,w;>0}}$
is obtained  from $\phi^{>0}_{\dot{\mathbf v}_+, \dot{\mathbf w}} \circ \bar \iota$
as in Lemma \ref{commutative}, where 
$\phi^{>0}_{\dot{\mathbf v}_+, \dot{\mathbf w}}$ is the 
corresponding parameterization in the related simply
laced group $\dot G$.
So we are looking at  parameterizations
of $\mathcal R_{v,w;>0}$ and $\mathcal R_{v',w;>0}$ embedded 
into $\dot {\mathcal R}_{ v, w}$ and $\dot {\mathcal R}_{ v', w}$,
respectively, which take the form
\begin{eqnarray*}
\mathbf t&\mapsto&\bar g_1 \dots \bar g_{k-1} y_{i_{k,1}}(t_{k}) y_{i_{k,2}}(t_{k}) \dots 
                  y_{i_{k,l}}(t_{k}) \dot s_{\bar i_{k+1}} \dots \dot s_{\bar i_m} \cdot B^+,\\ 
\mathbf t'&\mapsto &\bar g_1 \dots\bar g_{k-1} 
                  \dot s_{\bar i_{k}} \dot s_{\bar i_{k+1}} \dots \dot s_{\bar i_m} \cdot B^+,
\end{eqnarray*}                 
where $\dot s_{\bar i_k}=\dot s_{i_{k,1}} \dot s_{i_{k,2}} \dots \dot s_{i_{k,l}}$.

As before, there are 
unique positive integers $a_1,\dotsc, a_l$   
 such that 
$$f_{i_{k,1}}^{(a_1)} 
  \dots f_{i_{k,l}}^{(a_l)} \dot s_{\bar i_{k+1}} \dots \dot s_{\bar i_m} \cdot \eta = 
\dot s_{i_{k,1}} \dot s_{i_{k,2}} \dots 
                  \dot s_{i_{k,l}} \dot s_{\bar i_{k+1}} \dots \dot s_{\bar i_m} \cdot \eta,
                  $$ 
and for each $1 \leq h \leq l$, if we increase the corresponding exponent by $1$,  we have
$$f_{i_{k,h}}^{(a_h+1)} 
  \dots f_{i_{k,l}}^{(a_l)} \dot s_{\bar i_{k+1}} \dots \dot s_{\bar i_m} \cdot \eta =0.
$$
Now the composition
$i'\circ\phi_{\mathbf v_+,\mathbf w}^{>0}$ for $i'$ as in  Lemma~\ref{commutative} takes the form 
$$
\mathbf t\mapsto [p_1(\mathbf t),\dotsc, p_{N+1}(\mathbf t)],
$$
for polynomials $p_j$ with positive coefficients. And by the observation
about the $f_{i_{k,h}}$'s, 
the maximal power of $t_k$ in any of the $p_j$'s is $t_k^{a_1+\dotsc+a_l}$.

Finally we look at what happens if $t_k$ tends to infinity and repeat the 
arguments from the simply-laced case. In this case
$X_{\mathbf v'_+,\mathbf w}$ is the sub-toric
variety of $X_{\mathbf v_+,\mathbf w}$, which is given by those 
monomials which are multiples of $t_k^{a_1+\dotsc + a_l}$ 
(and other coordinates set to zero), and 
its moment polytope can be identified
with the face of $B_{\mathbf v_+,\mathbf w}$ cut out by the equations
$x_1=a_1,\dotsc, x_l=a_l$. Moreover from the analogue of \eqref{e:limitmap} it follows that 
the attaching map 
$\Phi_{\mathbf v'_+,\mathbf w}^{\ge 0}:X^{\ge 0}_{\mathbf v'_+,\mathbf w}\to 
\overline{\mathcal R^{>0}_{v',w}}$ 
is the restriction of
$\Phi_{\mathbf v_+,\mathbf w}^{\ge 0}$. This also implies that 
$B_{\mathbf v'_+,\mathbf w}$, which is isomorphic to 
$X^{\ge 0}_{\mathbf v'_+,\mathbf w}$, has codimension $1$ in   
$B_{\mathbf v_+,\mathbf w}$, making it a facet.
 \end{proof}

\begin{remark}\label{r:limit}
We note that in the situation of Proposition~\ref{almostregularface} we have also shown that
for any point $\chi(t_1,\dotsc, t_r)$ in 
$X^{>0}_{\mathbf v_+,\mathbf w}$, 
with $\chi$ as in Section \ref{Sec:Attach} \eqref{e:chi}, and any positive integer~$c$, the limit 
$\lim_{z\to\infty}\chi(t_1,\dotsc, t_{r-1},z^c t_r)$, lies in $X^{>0}_{\mathbf v'_+,\mathbf w}$.
\end{remark}
\begin{remark}\label{remains}
Proposition \ref{almostregularface} is a big step towards proving
Theorem \ref{regularity} for $(G/B)_{\geq 0}.$
To relate our notation to Definition \ref{regularface},
let $\tau=\mathcal R_{v,w;>0}$, $\sigma=\mathcal R_{v',w;>0}$,
$h=\Phi^{\geq 0}_{\mathbf v_+,\mathbf w}$, and 
$h'=\Phi^{\geq 0}_{\mathbf v'_+,\mathbf w}$.
By Proposition \ref{almostregularface}, 
$h^{-1}(\sigma)$ contains
$X^{>0}_{\mathbf v'_+, \mathbf w}$.  If we could show that
this is an equality, then because 
$h|_{X^{\geq 0}_{\mathbf v'_+, \mathbf w}}$ is the attaching map
$h'$, the restriction
$h|_{h^{-1}(\sigma)}=h|_{X^{> 0}_{\mathbf v'_+, \mathbf w}}$ would be a homeomorphism,
proving that Definition \ref{regularface} (1) is satisfied.
Furthermore, 
$\overline{h^{-1}(\sigma)} = \overline{X^{>0}_{v'_+,w}}$ is a 
closed ball of appropriate dimension, verifying (2).

\end{remark}

\begin{proposition}\label{step1}
Suppose that $w>v'\gtrdot v$ and we have a reduced expression
$\mathbf w = (i_1, \dotsc,i_j, \dotsc, i_m)$
such that $v'=s_{i_{j+1}}s_{i_{j+2}} \dots s_{i_m}$, and $v$ is obtained from 
$v'$ by removing a unique factor $s_{i_k}$ for $j+1 \leq k \leq m$.
Suppose furthermore that we have a sequence 
$(c_1,c_2,\dots, c_{j},c_k) \in \Z^{j+1}$ such that for $z>0$ and some/any 
fixed $t_1,t_2,\dots, t_j, t_k \in\R_{>0}$ the $1$-parameter family 
\begin{equation*}
g_z \cdot B^+:= y_{i_1}(z^{c_1} t_1) \dots y_{i_j}(z^{c_j} t_j) \dot s_{i_{j+1}} \dots
\dot s_{i_{k-1}} y_{i_k}(z^{c_{k}} t_{k}) \dot s_{i_{k+1}} \dots \dot s_{i_m} \cdot B^+
\end{equation*}
in $\mathcal R_{v,w;>0}$ tends as $z\to \infty$ to an element of $\mathcal R_{v',w;>0}$.
Then we must have $c_1 = \dots = c_j = 0$ and $c_{k} > 0$.
\end{proposition}

\begin{proof}
Recall that since $w$ ends
in $v'$ (that is, $\ell(w{v'}^{-1})=\ell(w)-\ell(v')$) we have a continuous map 
\begin{equation*}
\pi = \pi^w_{wv'^{-1}}: B^+ \dot w \cdot B^+ \to B^+ \dot w \dot v'^{-1} \cdot B^+.
\end{equation*}
See for example Section 4.3 of \cite{MarRie:ansatz}.
In terms of our parameterizations, if $x<w$ and we consider an 
element 
$$\phi_{\mathbf x_+,\mathbf w}(\mathbf t)=g_1\cdots g_m\cdot B^+\in
\mathcal R_{x,w}$$  
then $\pi(g_1\cdots g_m\cdot B^+)$ is just given by deleting the last 
$m-j$ factors spelling out the $v'$:
$$
\pi(g_1\cdots g_m\cdot B^+)=g_1\cdots g_j\cdot B^+.
$$
Note that in particular $\pi$ preserves total nonnegativity and takes
both $\mathcal R_{v',w;>0}$ and $\mathcal R_{v,w;>0}$ to the same cell, namely
$\mathcal R_{e,wv'^{-1};>0}$.

Now since the limit of $g_z \cdot B^+$ is assumed to lie in $\mathcal R_{v',w;>0}$
everything is taking place in $B^+ \dot w \cdot B^+$, the domain of $\pi$,
and we can apply $\pi$ to $g_z \cdot B^+$ before and after taking the limit
$z \to \infty$:

\begin{equation*}
\lim(\pi(g_z \cdot B^+)) = \pi(\lim(g_z \cdot B^+)) \in \mathcal R_{e,wv'^{-1};>0}.
\end{equation*}

So we see that 
\begin{equation*}
\pi(g_z \cdot B^+) = y_{i_1}(z^{c_1} t_1) \dots y_{i_j}(z^{c_j} t_j) \cdot B^+
\end{equation*}
is a 1-parameter family in $\mathcal R_{e,wv'^{-1};>0}$ whose limit point
as $z\to \infty$ again lies in $\mathcal R_{e,wv'^{-1};>0}$.  However, suppose that
one of the $c_1,\dots, c_j$ is nonzero.  Then taking the limit would certainly give
something that left the cell $\mathcal R_{e,wv'^{-1};>0}$ and went to a smaller one.
So all of these $c_i$ must be zero.

Given that the $c_i$ are zero for $i \leq j$, it is clear that $c_{k}$ must be 
positive for the limit of the original family to lie in $\mathcal R_{v',w;>0}$.
\end{proof}

\begin{remark} \label{r:family}  Let $C=(c_{1},\dots,c_{j},c_k)\in \Z^{j+1}$ be the
sequence from above. Then the 
$1$-parameter family in Proposition
\ref{step1} may also be written as  
$$g_z \cdot B^+=\phi^{>0}_{\mathbf v_+,\mathbf w}(z^C\cdot\mathbf t),$$
where $z>0$ and $z^C\cdot \mathbf t =(z^{c_{1}}t_{1},\dotsc, z^{c_{j}}t_j,z^{c_{k}}t_k)$,
 and where
$\phi^{>0}_{\mathbf v_+,\mathbf w}$ is the parameterization from
Section~\ref{s:param}.
\end{remark}

\begin{proposition}\label{step2}
Suppose that $w_0>v'\gtrdot v$, and choose any
reduced expression 
$\mathbf w_0 = (i_1, \dotsc, i_n)$.  
Let $\phi_{\mathbf v_+,\mathbf w_0}^{>0}$ 
and $\mathbf v^c_+$ be as defined in Section 3.2,
and write $\mathbf v^c_+ = \{h_1,\dots,h_r\}$ 
for $h_1<\cdots<h_r$.
There is a unique (up to positive scalar multiple)
choice of sequence $C=(c_{h_1},\dots, c_{h_r}) 
\in \Z^{r}$ such that, 
for $z>0$ and 
for some/any 
fixed $t_{h_1}, \dots ,t_{h_r}  \in \R_{>0}$ the $1$-parameter family in $\mathcal R_{v,w_0;>0}$,
$$
g_z\cdot B^+=\phi^{>0}_{\mathbf v_+,\mathbf w_0}(z^C\cdot\mathbf t),
$$
 tends as $z\to \infty$ to an element of $\mathcal R_{v',w_0;>0}$.
Here $z^C\cdot \mathbf t =(z^{c_{h_1}}t_{h_1},\dotsc, z^{c_{h_r}}t_{h_r})$.
\end{proposition}

\begin{proof}
Let us first assume that
$\mathbf w_0 = (i_1, \dotsc ,i_n)$ ends with a reduced expression for $v'$.
Then we are in the situation of Proposition \ref{step1}, and so in terms of the coordinates
of the parameterization $\phi^{>0}_{\mathbf v_+,\mathbf w'_0}$, there is a unique 
vector $C\in \mathbb Z^{r}$, up to positive scalar multiple, giving a 1-parameter family $g_z\cdot B^+=
\phi^{>0}_{\mathbf v_+,\mathbf w_0}(z^C\cdot\mathbf t)$, 
whose limit point lies in $\mathcal R_{v',w_0;>0}$.

Recall that any two reduced expressions for $w_0$ can be related by braid and 
commuting relations, and suppose now that $\mathbf w_0$ is an arbitrary reduced expression 
for $w_0$. It suffices to prove that if the statement of the Proposition holds
for $\mathbf w_0$ then it also holds for any $\mathbf w_0'$ obtained from 
$\mathbf w_0$ by a braid relation or commuting relation.  

This is obvious in the case of a commuting relation. 
Now suppose $\mathbf w_0$ 
and $\mathbf w_0'$ are related by a more general braid relation,
and $C$ is the vector associated (up to positive
scalar multiple) to $\mathbf w_0$. 
The braid relation gives us a change of coordinates
$\kappa(\mathbf t)=\mathbf t'$ 
which is rational, homogeneous and subtraction-free,
see Section \ref{s:coord}. 
We let $\mathbf t'_z:=\kappa(z^{C}\cdot \mathbf t)$.
For example, if $w_0$ is the longest element of the symmetric 
group $S_3$, then applying Formula (1) of Section \ref{s:coord} to 
$y_1(z^{c_1} t_1) y_2(z^{c_2} t_2) y_3(z^{c_3} t_3) \cdot B^+$,
gives 
\begin{equation*}y_2\left(\frac{z^{c_2+c_3} t_2 t_3}{z^{c_1}t_1+z^{c_3}t_3}\right)\
 y_1\left(z^{c_1}t_1+z^{c_3}t_3\right) \
 y_2\left(\frac{z^{c_1+c_2}t_1 t_2}{z^{c_1}t_1+z^{c_3}t_3}\right)\cdot B^+,
\end{equation*}
and the  entries are the components of $\mathbf t'_z$. Because the 
components of $\mathbf t'_z$ are subtraction-free, the maximal power
of $z$ in each one dominates the limit as $z\to\infty$.
In this example, we have therefore
\begin{multline*}
\lim_{z\to \infty}(\phi^{>0}_{1,s_2s_1s_2}(\mathbf t'_z))\\ =
\lim_{z\to\infty}(\phi^{>0}_{1,s_2s_1s_2}(z^{c_2+c_3-\max(c_1,c_3)}
q_1(\mathbf t),z^{\max(c_1,c_3)}q_2(\mathbf t),z^{c_1+c_2-\max(c_1,c_3)}q_3(\mathbf t)))
\end{multline*}
where the $q_i(\mathbf t)$ are new rational, subtraction-free functions in the $t_j$'s. We define $C':=(c_2+c_3-\max(c_1,c_3),\max(c_1,c_3),c_1+c_2-\max(c_1,c_3))$.

The same procedure can be applied in the general case 
to define a $C'$ out of the original $C$ as well as new 
rational, subtraction-free functions $q_i$
defining $q(\mathbf t)=(q_1(\mathbf t),\dotsc, q_r(\mathbf t))$
such that 
\begin{equation*}
\lim_{z\to \infty}(\phi^{>0}_{\mathbf v_+,\mathbf w_0'}(\mathbf t'_z)) =
\lim_{z\to\infty}\left (\phi^{>0}_{\mathbf v_+,\mathbf w_0'}\left (z^{C'}\cdot q(\mathbf t)\right )\right).
\end{equation*}
We note that $C'$ is also in general related to the original $C$
by a piecewise linear transformation, which one may 
compare to the {\it zones}
and the map $R_j^{j'}$ from Sections~9.1 and 9.2 of 
\cite{Lusztig3}.

We now see by testing out on a point 
$\phi^{>0}_{\mathbf v_+,\mathbf w_0'}(\mathbf t')\in \mathcal R_{v,w_0;>0}$, 
where $\mathbf t'=q(\mathbf t)$, that 
$$
\lim_{z\to\infty}(\phi^{>0}_{\mathbf v_+,\mathbf w_0'}(z^{C'}\cdot \mathbf t'))=
\lim_{z\to \infty}(\phi^{>0}_{\mathbf v_+,\mathbf w_0'}(\mathbf t'_z)) =
\lim_{z\to \infty}(\phi^{>0}_{\mathbf v_+,\mathbf w_0}(z^C\cdot\mathbf t)), 
$$
and lies in $\mathcal R_{v',w_0;>0}$. Therefore we have 
found a $C'\in \mathbb Z^r$ with the required property 
for our new reduced 
expression $\mathbf w_0'$.  

To prove uniqueness, suppose $D'\in \mathbb Z^r$ is a different element such that 
$$
\lim_{z\to\infty}(\phi^{>0}_{\mathbf v_+,\mathbf w_0'}(z^{D'} \cdot\mathbf t'))\in \mathcal R_{v',w_0;>0}
$$
 for some $\mathbf t'\in \R_{>0}^{r}$. Then we may apply the coordinate transformation
 back from the reduced expression $\mathbf w_0'$ to $\mathbf w_0$. Thus $D'$ is
 transformed by a piecewise-linear transformation to a $D\in \Z^r$ 
 such that  
$$
\lim_{z\to\infty}(\phi^{>0}_{\mathbf v_+,\mathbf w_0}(z^{D} \cdot\mathbf t))\in \mathcal R_{v',w_0;>0}.
$$
However this implies that $D$ is a positive multiple of $C$, and 
by applying the original transformation again, that $D'$ was a positive multiple of $C'$.
\end{proof}

\begin{proposition}\label{step3}
Choose $w> v' \gtrdot v$ and 
a reduced expression $\mathbf w = (i_1, \dotsc, i_m)$.
Let $\phi_{\mathbf v_+,\mathbf w}^{>0}$ 
and $\mathbf v^c_+$ be as defined in Section 3.2,
and write $\mathbf v^c_+ = \{h_1,\dots,h_r\}$ 
for $h_1<\cdots<h_r$.
 Suppose  in addition that $\mathbf v'_+$ is equal to 
$\mathbf v_+ \cup \{h_r \}$.
For $C=(c_{h_1},\dots,c_{h_r})\in \Z^r$ and $z>0$
we consider
the $1$-parameter family 
$$
g_z \cdot B^+=\phi_{\mathbf v_+,\mathbf w}^{>0}(z^{C}\cdot \mathbf t)
$$ 
in $\mathcal R_{v,w}^{>0}$, where $z^C\cdot \mathbf t =(z^{c_{h_1}}t_{h_1},\dotsc, z^{c_{h_r}}t_{h_r})$.
Then if $g_z \cdot B^+$ 
in $\mathcal R_{v,w;>0}$ tends as $z\to \infty$ to an element of $\mathcal R_{v',w;>0}$, 
we must have $c_{h_1} = \dots = c_{h_{r-1}} = 0$ and $c_{h_r} >0$.
\end{proposition}

\begin{remark} Note that the condition on $\mathbf v_+'$ in Proposition~5.13 is 
equivalent to the pair of cells $\mathcal R_{v,w;>0}\gtrdot \mathcal R_{v',w;>0}$ being 
good with respect to $\mathbf w$ in the sense of Theorem~\ref{regularity}.
\end{remark}
\begin{proof}
Suppose $C=(c_{h_1},\dotsc, c_{h_r})\in\Z^r$
has the property,
\begin{equation}\label{e:limit}
g_z\cdot B^+ \text{ has limit in $\mathcal R_{v',w;>0}$ as $z\to \infty$.}
\end{equation}
Choose a reduced expression $\mathbf w_0=(j_1, \dotsc, j_{n-m}, i_1, \dots , {i_m})$ 
for $w_0$ ending with $\mathbf w$, and let us fix $u_1,\dotsc, u_{n-m}\in\R_{>0}$. Then 
we obtain a new one-parameter family,
$$
y_{j_1}(u_1) \dots y_{j_{n-m}}(u_{n-m}) g_z \cdot B^+,
$$
which lies in $\mathcal R_{v,w_0;>0}$ for $z>0$ and tends to an element in 
$\mathcal R_{v',w_0;>0}$ as $z\to \infty$. Now Proposition \ref{step2} is applicable
and we have that $\tilde C=(0,\dotsc, 0, c_{h_1},c_{h_2},\dotsc, c_{h_r})$ is the unique
(up to positive scalar multiple) choice of $\tilde C\in \Z^{n-m+r}$
such that the corresponding 1-parameter family in  $\mathcal R_{v,w_0;>0}$
tends to a point in $\mathcal R_{v',w_0;>0}$. It follows that the original 
$r$-tuple $(c_{h_1},\dotsc, c_{h_r})$
satisfying \eqref{e:limit} is also uniquely determined up to positive
scalar multiple.

Now it only remains to prove that \eqref{e:limit} holds for 
$(c_{h_1},\dotsc, c_{h_{r-1}},c_{h_r})=(0,\dotsc, 0,1)$.
But this is clear, by the same argument  we used for
 \eqref{e:limitmap} in the proof of
 Proposition~\ref{almostregularface}.\end{proof}

We now turn to the proof of Theorem \ref{regularity}

\begin{proof}[Proof of Theorem~\ref{regularity}]
We begin with the full flag variety case.
Recall the natural inclusion $X^{>0}_{\mathbf v'_+, \mathbf w}\hookrightarrow X^{\geq 0}_{\mathbf v_+, \mathbf w}$, given by Proposition~\ref{almostregularface}, for which
$\Phi^{\geq 0}_{\mathbf v_+ ,\mathbf w}( X^{>0}_{\mathbf v'_+, \mathbf w})=\mathcal R_{v',w;>0}$.
By Remark \ref{remains}, it suffices to prove the claim that
\begin{equation}\label{e:preimage}
(\Phi^{\geq 0}_{\mathbf v_+,\mathbf w})^{-1}(\mathcal R_{v',w;>0})=X^{>0}_{\mathbf v'_+, \mathbf w}.
\end{equation}

Suppose we have $x'\in X^{\geq 0}_{\mathbf v_+, \mathbf w}$ such that
 $\Phi^{\geq 0}_{\mathbf v_+,\mathbf w}(x')\in\mathcal R_{v',w;>0}$.
 We can approach $x'$ from a point  in the interior, $X^{>0}_{\mathbf v_+, \mathbf w}$,
 by a 1-parameter family. 
 Namely,  
 $$
 x'=\lim_{z\to \infty}\chi(z^C\cdot\mathbf t)=
 \lim_{z\to \infty}\chi(z^{c_{1}} t_{1},\dotsc, z^{c_{r}} t_{r}),
 $$
for some $\mathbf t\in\R_{>0}^r$, $C\in\Z^{r}$ and $\chi$ the map 
from \eqref{e:chi} associated to $X_{\mathbf v_+,\mathbf w}$.  
Therefore
$$
\Phi^{\geq 0}_{\mathbf v_+,\mathbf w}(x')= 
\lim_{z\to\infty}\Phi^{>0}_{\mathbf v_+,\mathbf w}
(\chi(z^C\cdot\mathbf t))=\lim_{z\to\infty}
\phi_{\mathbf v_+,\mathbf w}^{>0}(z^{C}\cdot \mathbf t),
$$
where the $1$-parameter subgroup on the right hand side is 
as in Proposition~\ref{step3}.
Since by our assumption $\Phi^{\geq 0}_{\mathbf v_+,\mathbf w}(x')\in
\mathcal R_{v',w;>0}$, and we are in the `good' situation,
Proposition~\ref{step3} tells us that 
$C=(0,\dotsc,0,c)$, for positive $c$. But this implies
that 
$$
x'=\lim_{z\to\infty}\chi(z^C\cdot \mathbf t)\in X_{\mathbf v'_+,\mathbf w}^{>0},
$$
see Remark~\ref{r:limit}. Therefore the claim \eqref{e:preimage} holds
and the Theorem is true for the full flag variety.

Now consider the case of $G/P_J$.  We have that 
$\pi^J$ gives an isomorphism from $\mathcal R_{xu^{-1},w;>0}$ to $\mathcal P_{x,u,w;>0}^J$ and from
$\mathcal R_{x'{u'}^{-1},w;>0}$ to $\mathcal P_{x',u',w;>0}^J$.  
We've already proved that 
$\mathcal R_{x'{u'}^{-1},w;>0}$ is a regular face of 
$\mathcal R_{x{u}^{-1},w;>0}$ with respect to the attaching map 
$\Phi^{\geq 0}_{\mathbf{x u^{-1}_+},\mathbf w}$.
Recall that the attaching map for $\mathcal P_{x,u,w;>0}$ is simply
$\pi^J \circ
\Phi^{\geq 0}_{\mathbf{x u^{-1}_+},\mathbf w}$.

When we restrict 
$\pi^J$ to $\overline{\mathcal R_{xu^{-1},w;>0}}$, it is straightforward to check that the preimage 
of $\mathcal P_{x',u',w;>0}^J$ is $\mathcal R_{x'{u'}^{-1},w;>0}$.  By Theorem 
\ref{regularity}  in the full flag variety case, we know that
$(\Phi^{\geq 0}_{\mathbf{xu^{-1}_+},\mathbf w})^{-1}(\mathcal R_{x'{u'}^{-1},w;>0})=
X^{>0}_{\mathbf{x'{u'}^{-1}_+}, \mathbf w}$ and 
$\Phi^{\geq 0}_{\mathbf{xu^{-1}_+},\mathbf w}$ is a homeomorphism from 
$X^{>0}_{\mathbf{x'{u'}^{-1}_+}, \mathbf w}$ to $\mathcal R_{x'{u'}^{-1},w;>0}$.
Therefore 
$(\pi^J \circ \Phi^{\geq 0}_{\mathbf{xu^{-1}_+},\mathbf w})^{-1}(\mathcal R_{x'{u'}^{-1},w;>0})=
X^{>0}_{\mathbf{x'{u'}^{-1}_+}, \mathbf w}$ and 
$\pi^J \circ \Phi^{\geq 0}_{\mathbf{xu^{-1}_+},\mathbf w}$ is a homeomorphism from 
$X^{>0}_{\mathbf{x'{u'}^{-1}_+}, \mathbf w}$ to $\mathcal R_{x'{u'}^{-1},w;>0}$.
It follows that 
$\mathcal P_{x',u',w; >0}^J$ is a regular face of 
$P_{x,u,w; >0}^J$ with respect to the attaching map 
$\pi^J \circ
\Phi^{\geq 0}_{\mathbf{x u^{-1}_+},\mathbf w}$.
\end{proof}

\section{Preliminaries on poset topology}\label{PosetTopology}

\subsection{Preliminaries}
Poset topology is the study of combinatorial properties of a partially
ordered set, or {\it poset}, which reflect the topology of an
associated simplicial or cell complex.  
In this section we will
review some of the basic definitions and results of poset topology.

Let $P$ be a poset with order relation $<$.  We will use the symbol
$\lessdot$ to denote a {\it covering relation} in the poset: $x \lessdot
y$ means that $x < y$ and there is no $z$ such that $x < z < y$.
Additionally, if $x < y$ then $[x,y]$ denotes the {\it closed
interval}
from $x$ to $y$; that is, the set $\{z \in P \ \vert \ x \leq z \leq
y \}$.

We will often identify a poset $P$ with its {\it Hasse diagram}, which
is the graph whose vertices represent elements of $P$ and whose 
edges depict covering relations.

The natural geometric object associated to a poset $P$ is
the realization of its {\it order complex} (or {\it nerve}). The
order complex $\Delta(P)$ is the simplicial complex
whose vertices are the elements of $P$ and whose simplices are the
chains $x_0 < x_1 < \dots < x_k$ in $P$.  

A poset is called
{\it bounded} if it has a least element $\hat{0}$ and a greatest element
$\hat{1}$.  
The {\it atoms} of a bounded poset are the elements
which cover $\hat{0}$.  Dually, the {\it coatoms} are the elements which
are covered by $\hat{1}$.
A finite
poset is said to be {\it pure} if all maximal
chains have the same length, and {\it graded}, if in addition, it is
bounded.  An element $x$ of a graded poset $P$ has a well-defined
{\it rank} $\rho(x)$ equal to the length of an
unrefinable chain from $\hat{0}$ to $x$ in $P$.
A poset $P$ is called {\it thin} if every interval of length $2$
is a {\it diamond},
i.e. if for any $p < q$ such that $\rank(q)-\rank(p) = 2$,
there are exactly two elements in the open interval $(p,q)$.

\subsection{Shellability and edge-labelings}
A pure finite simplicial complex $\Delta$ is said to be {\it shellable}
if its maximal faces can be ordered $F_1, F_2, \dots , F_n$ in such
a way that
$F_k \cap (\cup_{i=1}^{k-1} F_i)$ is a nonempty union of maximal proper
faces of $F_k$ for $k=2, 3, \dots , n$.
Certain edge-labelings of posets can be used to prove that
the corresponding order complexes are shellable. These
techniques were pioneered by Bjorner \cite{B1}, 
and Bjorner and Wachs \cite{BW1}.

One technique that can be used to prove that an order complex
$\Delta(P)$ is shellable is the notion of
{\it lexicographic shellability}, or
{EL-shellability}, which was first introduced by Bjorner \cite{B1}.
Let $P$ be a graded poset, and let
$\E(P)$ be the set of edges of the Hasse diagram of $P$, i.e.
$\E(P) = \{(x,y) \in P \times P\ \vline \ x \gtrdot y\}$.
An {\it edge labeling} of $P$ is a map
$\lambda : \E(P) \to \Lambda$ where $\Lambda$ is some poset
(usually the integers).  Given an edge labeling $\lambda$,
each maximal chain $c = (x_0 \gtrdot x_1 \gtrdot \dots \gtrdot x_k)$ of
length $k$ can be associated with a $k$-tuple
$\sigma(c) = (\lambda(x_0,x_1), \lambda(x_1, x_2), \dots ,
\lambda(x_{k-1}, x_k))$.  We say that $c$ is an {\it increasing
chain} if the $k$-tuple $\sigma(c)$ is increasing; that is,
if $\lambda(x_0, x_1) \leq \lambda(x_1, x_2) \leq \dots \leq
\lambda(x_{k-1}, x_k)$.  The edge labeling allows us
to order the maximal chains of any interval of $P$ by
ordering the corresponding $k$-tuples lexicographically.
If $\sigma(c_1)$ lexicographically precedes
$\sigma(c_2)$ then we say that $c_1$ lexicographically precedes
$c_2$ and we denote this by $c_1  <_L c_2$.

\begin{definition}
An edge labeling is called an {\it EL-labeling} ({\it edge lexicographical
labeling}) if for every interval $[x,y]$ in $P$,
\begin{enumerate}
\item there is a unique increasing maximal chain $c$ in $[x,y]$, and
\item $c <_L c'$ for all other maximal chains $c^{\prime}$ in $[x,y]$.
\end{enumerate}
\end{definition}

If one has an EL-labeling of $P$, it is not hard to see 
that the corresponding order on maximal chains gives a shelling of the order complex \cite{B1}.
Therefore a
graded poset that admits an EL-labeling is said to be {\it EL-shellable}.

Given an EL-labeling $\lambda$ of $P$ and $x\in P$, we define
$\Last_{\lambda}(x)$ to be the set of elements $z\lessdot x$ such that 
$\lambda(z \lessdot x)$ is maximal among the set
$\{ \lambda(y\lessdot x) \ \vert \ y \lessdot x \}$.

\subsection{Face posets of cell complexes}\label{s:CW}
When analyzing a CW complex $\K$, 
it is sometimes useful to 
study its {\it face poset} $\F(\K)$, 
as in Definition \ref{faceposet}.
The face poset is a natural poset to study
particularly if the CW complex has the {\it subcomplex property}, 
i.e. if
the closure of a cell is a union of cells.  

The class of regular CW complexes is particularly nice.
Recall that 
a CW complex is {\it regular} if the closure
of each cell  is homeomorphic to a closed ball and if
additionally the closure minus the interior of a 
cell is homeomorphic to a sphere.
In general, the order complex $\Vert
\F(\K)  \Vert$ does not reveal the topology of $\K$.
However, the following result shows that regular CW complexes are
combinatorial objects in the sense that the incidence relation of
cells determines their topology.

\begin{proposition} \cite[Proposition~4.7.8]{RedBook}\label{RedTheorem}
Let $\K$ be a regular CW complex.  Then
$\K$ is homeomorphic to
$\Vert \F(\K)  \Vert$.
\end{proposition}

We will call a poset $P$ a {\it CW poset} if it is the face poset of 
a regular CW complex.

There is a notion of shelling for regular
cell complexes (which is distinct from the notion of shelling of the 
order complex), due to Bjorner and Wachs.  Such a shelling is a certain ordering
on the coatoms of the face poset.  We don't need the precise definition, only 
the following result that an EL-labeling of the augmented face poset of a 
regular cell complex gives rise to a shelling.

\begin{theorem}\cite[Theorem 5.11]{BW3}\cite[Theorem 13.2]{BW4}\label{CL-RAO}
If $P$ is the augmented face poset of a finite-dimensional 
regular CW complex $K$, then any EL-labeling
of $P$ gives rise to a shelling of $K$.  To go from the EL-labeling to the shelling
one chooses the ordering on coatoms which is specified by
the order on edges between the unique greatest element and the coatoms.
\end{theorem}

\section{Discrete Morse theory for general CW complexes}\label{sec:DMT}

In this section we review Forman's powerful {\it discrete 
Morse theory} \cite{RF}.  The theory comes in three ``flavors": 
for simplicial complexes,  regular CW complexes, and
general CW complexes.  In each setting, one needs to 
find a certain {\it discrete Morse function}, and then
the main 
theorem says that the space in question
is homotopy equivalent to another simpler space obtained by 
collapsing non-critical cells.

The first two flavors of the theory are the 
simplest and most widely used, 
because in these two settings
a result of Chari
\cite{Chari} implies that a 
discrete Morse function is equivalent to
a matching on the face poset of the CW
complex.  To work with the third flavor of the theory,
one must check 
some additional technical
conditions: the {\it discrete Morse hypothesis}, as well 
as an extra topological condition included 
in the definition of discrete Morse function.
However, as we will see in Theorem \ref{MainMorseTheorem}, it is 
enough to find a matching on the face poset of a CW complex 
with the {\it subcomplex property}
such that matched edges are {\it regular}.  Although this 
result will not be 
surprising to the experts,  we
could not find it in the literature and so we give an 
exposition here.
The proof follows from 
an argument of Kozlov \cite[Proof of Theorem 3.2]{Kozlov}.\footnote{Although
Theorem 3.2 of \cite{Kozlov} was in the more restricted setting
of regular CW complexes, the proof still holds in our situtation.}

\subsection{Forman's discrete Morse theorem for general CW complexes}

Let $K$ be a finite CW complex and let $Q$ be its poset of cells.
Recall the definition of regular face from Definition \ref{regularface}.

\begin{definition}\cite[p. 102]{RF} \label{dMf}
A function $f:Q\to \R$ is a discrete Morse function if for every
$\sigma^{(p)}$ of dimension $p$, the following conditions hold:
\begin{enumerate}
\item
$\# \{ \tau^{(p+1)} \vert \tau^{(p+1)}>\sigma \text{ and } f(\tau) \leq f(\sigma) \} \leq 1$
\item
$\# \{ v^{(p-1)} \vert v^{(p-1)} <\sigma \text{ and }
f(v) \geq f(\sigma) \} \leq 1$.
\item If $\sigma$ is an irregular face of
$\tau^{(p+1)}$ then $f(\tau)>f(\sigma)$.
\item If $v^{(p-1)}$ is an irregular face of $\sigma$
then $f(v)<f(\sigma)$.
\end{enumerate}
\end{definition}

Note that a Morse function is a function
 which is ``almost increasing".  Indeed, one should think 
of a Morse function as a function which specifies the order
in which to attach the cells of a homotopy-equivalent CW complex
\cite{Kozlov}.

\begin{definition}\label{critical}
We say that a cell $\sigma^{(p)}$ is critical if 
\begin{enumerate}
\item $\# \{ \tau^{(p+1)}>\sigma \ \vert\ f(\tau) \leq f(\sigma) \} =0$, and
\item $\# \{ v^{(p-1)}<\sigma \ \vert\ f(v) \geq f(\sigma) \} =0$.
\end{enumerate}
\end{definition}

Let $m_p(f)$ denote the number of critical cells of dimension $p$.

For each cell $\sigma$ of a CW complex $K$,
let $\Carrier(\sigma)$ denote the smallest subcomplex
of $K$ containing $\sigma$.  
If $K$ has the subcomplex property (see Section \ref{s:CW}),
then for any $\sigma$, 
$\Carrier(\sigma)$ is its closure, and hence
condition (1) of 
Definition \ref{dMh} below is satisfied.

\begin{definition}\cite[p. 136]{RF} \label{dMh}
Given a CW complex $K$ and a discrete Morse function $f$, we say that
$(K,f)$ satisfies the {\it Discrete Morse Hypothesis} if:
\begin{enumerate}
\item For every pair of cells $\sigma$ and $\tau$, if
$\tau \subset \Carrier(\sigma)$ and $\tau$ is not a
face of $\sigma$, then $f(\tau) \leq f(\sigma)$.
\item Whenever  there is a
$\tau > \sigma^{(p)}$ with
$f(\tau) < f(\sigma)$ then there is a
$\tilde{\tau}^{(p+1)}$ with $\tilde{\tau}>\sigma$
and $f(\tilde{\tau}) \leq f(\tau)$.
\end{enumerate}
\end{definition}

The following is Forman's main theorem for general CW complexes.

\begin{theorem}\cite[Theorem 10.2]{RF}\label{FormanTheorem}
Let $K$ be a CW complex satisfying the Discrete Morse
Hypothesis, and $f$ a discrete Morse function.
Then $K$ is homotopy equivalent to a CW complex with $m_p(f)$
cells of
dimension $p$.
\end{theorem}

\subsection{Discrete Morse functions as matchings}

Chari \cite{Chari}
pointed out that when the CW complex is regular,
one can depict a Morse function $f$
as a certain kind of
matching on the Hasse diagram of the poset of cells.
Given such an $f$, we define a matching $M(f)$ on the Hasse diagram of $Q$
whose edges correspond to the pairs of cells in which we get equality in
(1) or (2) of Definition \ref{dMf}.

Recall that a {\it matching} of a graph $G=(V,E)$
is a subset $M$ of edges of $G$ such that each vertex in $V$ is 
incident to at most one edge of $M$.
We define a {\it Morse matching} $M$ on a poset $Q$ to be a matching on the 
Hasse diagram such that if edges in $M$ are directed from lower to 
higher-dimension
elements and all other edges are directed from higher to lower-dimension elements,
then the resulting directed graph $G(M)$ is acyclic.
We refer to any elements of $Q$ which are not matched by $M$ as {\it critical elements}
(or {\it critical cells}).

In the situation of {\it arbitrary} CW complexes, a Morse matching 
such that matched edges are regular gives rise to a discrete Morse function satisfying
property (2) of the Discrete Morse Hypothesis, as the following lemma shows.
The proof of this lemma follows an argument of Kozlov
\cite[Proof of Theorem 3.2]{Kozlov}.

\begin{lemma}\label{KozlovLemma}
Let $M$ be a Morse matching on the face poset $Q$ of a CW complex,
such that each edge in $M$ corresponds to a regular pair 
of faces in the CW complex.
Then there exists a discrete Morse function $f_M$, satisfying 
property (2) of the Discrete Morse Hypothesis, whose critical cells are exactly
the critical cells of $M$.
\end{lemma}

\begin{proof}
We will inductively assign positive integer labels to each of 
the elements of $Q$,
producing a function $f_M$.  
Moreover, $f_M$ will have the property that
if $x<y$, $f_M(x) \leq f_M(y)$, 
with $f_M(x)=f_M(y)$ if and only if $(x,y)\in M$; 
in the case that $(x,y)\in M$, we will label $x$ and $y$ at the same time. 

At each step, let $x$ be one of the elements of $Q$ of 
minimal rank (dimension) among 
those not yet labeled, and let $i$ be the smallest positive integer not yet
appearing as a label in $Q$.  If $x$ is not in $M$ and hence critical, label $x$
with $i$.  If $x$ is not critical, then we must have $(x,y)\in M$, where $x \lessdot y$.
If each $z < y$ in $Q$ is labeled, then label both $x$ and $y$ with $i$.  Otherwise,
there exists $x_1<y$ in $Q$ where $x_1$ is not labeled; repeat the argument with $x_1$
taking the place of $x$.  Either we will label $x_1$ or a pair $(x_1,y_1)\in M$, 
or, since $G(M)$ is acyclic,
we will find $x_2 \neq x$, $x_2 \neq x_1$, $y_1 > x_2$, etc.

Since there are finitely many elements of $Q$, the process will terminate.  Since we
never label an element $y\in Q$ until we have labeled each $x<y$, $f_M$ has the 
property that for $x<y$, $f_M(x) \leq f_M(y)$.  Therefore condition (2) of the 
Discrete Morse Hypothesis is satisfied.  The only case in which
$f_M(x) = f_M(y)$ is when $(x,y)\in M$, i.e. $(x \lessdot y)$ is a regular pair of faces  --
and so conditions (3) and (4) of Definition \ref{dMf} are satisfied.  
Conditions (1) and (2) of Definition \ref{dMf} are satisfied because $M$ is
a matching.  Finally, it is clear 
that the cells which are critical with respect to $M$
are exactly those which are critical with respect to Definition \ref{critical}.
\end{proof}

We now restate
Forman's Morse Theorem for general CW complexes 
in terms of Morse matchings.

\begin{theorem}\label{MainMorseTheorem}
Let $K$ be a CW complex with the subcomplex property.
Suppose its face poset $Q$ has a Morse matching $M$, such that  
whenever $(\sigma^{(p)}, \tau^{(p+1)}) \in M$, $\sigma$ is a regular face of $\tau$.
Let $m_p(M)$ denote the number of critical cells of dimension $p$.
Then $K$
is homotopy equivalent to a CW complex 
with  $m_p(M)$
cells of dimension $p$.
\end{theorem}

\begin{proof}
By Lemma \ref{KozlovLemma}, we have a discrete Morse function $f$ for $K$
satisfying condition (2) of Definition \ref{dMh}.
Since $K$ has the subcomplex property, condition (1) of Definition \ref{dMh} 
is satisfied.  
The result now follows from Theorem \ref{FormanTheorem}.
\end{proof}

\subsection{From edge-labelings to Morse matchings}\label{sec:EdgeToMatch}

Both lexicographic shellability and discrete Morse theory are combinatorial
tools which can be used to investigate the topology of a CW complex. 
In this section we will recall a result of Chari \cite{Chari}, which 
proves the existence of a certain Morse matching given a 
shelling
of a regular CW complex.  We will translate this into a statement
about constructing a Morse matching from 
an EL -labeling, and note
that one can gain some fairly explicit information 
about the Morse matching from 
the EL-labeling.

Recall the notion of pseudomanifold, e.g. from \cite{Chari}.  Note that 
by a result of Bing \cite[Chapter 4]{RedBook}, a 
shellable pseudomanifold is in particular a regular CW complex which is either 
a ball or a sphere.

\begin{proposition}\cite[Proposition 4.1]{Chari}\label{ChariMorse}
Let $\sigma_1, \sigma_2, \dots ,\sigma_m$ be a shelling of a
$d$-pseudomanifold $\Sigma$
and let $v$ be any vertex in $\overline{\sigma_1}$.
Then the face poset of $\Sigma$ admits a Morse matching $M$
such that:
\begin{itemize}
\item 
If $\Sigma$ is the $d$-sphere then $v$ and $\sigma_m$ are the only
critical cells, while if $\Sigma$ is a $d$-ball,
then $v$ is the only critical cell.
\end{itemize}
\end{proposition}

By Theorem \ref{CL-RAO}, an EL-labeling of the augmented face 
poset of a pseudomanifold gives rise
to a shelling.
Chari used induction to 
construct the Morse matching of Proposition \ref{ChariMorse}.   
Chari's proof of Proposition \ref{ChariMorse}, applied to a shelling
which comes from an EL-labeling, implies the following.

\begin{corollary}\label{Last}
Suppose that $\lambda$ is an EL-labeling of the augmented face poset $Q$ of a pseudomanifold
$\Sigma$.
Let $M_{\lambda}$ be the Morse matching given by Proposition \ref{ChariMorse}.
Every $(\sigma \lessdot \tau) \in M$ has the following property:
$\sigma \in \Last_{\lambda}(\tau)$.
\end{corollary}

In fact, when the edge labels in $\lambda$ 
come from a totally
ordered set, Chari's proof of Proposition \ref{ChariMorse}
gives  the following  algorithm
for obtaining the Morse matching.

\begin{corollary} \label{algo}
Suppose that $\lambda$ is  an EL-labeling of the augmented face poset $Q$ of a pseudomanifold
$\Sigma$.
Then the Morse matching  $M_{\lambda}$  given by
Proposition \ref{ChariMorse} can be constructed
as follows.
Set $n$ equal to the rank of the poset $Q$ and set $M = \emptyset$.
\begin{enumerate}
\item Consider all unmatched elements $\sigma$ of rank $n$,
      and for each, add $\Last_{\lambda}(\sigma) \lessdot \sigma$ to $M$.
\item Decrease $n$ by 1 and go to step 1.
\end{enumerate}

\end{corollary}

\begin{remark}
Chari also extends Proposition \ref{ChariMorse} to regular 
CW complexes \cite[Theorem 4.2]{Chari}.  Corollaries 
\ref{Last} and \ref{algo} also hold in this situation.
\end{remark}

\begin{remark}
Proposition \ref{ChariMorse} and Corollary \ref{Last} can 
be useful even when $K$ is a CW complex not known to be regular.
In particular, if the face poset $Q$ of $K$ is a CW poset, 
then there exists a regular CW complex $K_{reg}$ whose face poset
is $Q$.  Therefore
one can still use these
results to construct a Morse matching of $K$.  
\end{remark}

\section{The Bruhat order, shellability, and reduced expressions}\label{Bruhat}

Fix a Coxeter system $(W,I)$ 
and let $T$ be the set of  reflections. 
In this section we will review some properties of the Bruhat order $\leq$
and  prove a result (Proposition \ref{prop:reduced}) about reduced expressions
which will be needed for the proof of Proposition \ref{smallMorse}.

The first part of Theorem \ref{BruhatBall} below is due to 
Bjorner and Wachs \cite{BW1}.  The second part follows from 
the first  together with 
Bjorner's result \cite{B2} characterizing CW posets.

\begin{theorem} \cite{BW1} \cite{B2} \label{BruhatBall}
The Bruhat order of a Coxeter group is thin and (CL)-shellable.
Furthermore, an interval with at least
two elements 
is the augmented face poset of a regular CW complex homeomorphic
to a ball.
\footnote{
Recall from Remark \ref{augmented} that 
\cite{BW1} augments the poset of cells
with a $\hat{0}$ and also a greatest element $\hat{1}$.
Using this convention, \cite{BW1}  considers
intervals in Bruhat order to be posets associated to regular CW complexes 
homeomorphic to {\it spheres}. } 
\end{theorem}

Theorem \ref{BruhatBall} together with Proposition \ref{ChariMorse}
imply the following.

\begin{corollary}\label{Bruhat-Matching}
Let $v < w$ in the Bruhat order of a Coxeter group.
Then if we remove $v$ from the interval $[v,w]$
(which plays the role of $\hat{0}$), 
there is a Morse matching $M$ on the resulting poset with 
one critical element $u$ of minimal rank.
Adding $v$ back to the poset and adding the edge $(v,u)$ to $M$, 
we get a Morse matching on $[v,w]$ with no critical elements.
\end{corollary}

Dyer \cite{Dyer} subsequently strengthened the Bjorner-Wachs result 
by giving an EL-labeling of Bruhat order.
Dyer's primary tool was his notion of ``reflection orders," certain
total orderings of $T$  which can be characterized as follows. 

\begin{definition} \cite[Proposition 2.13]{Dyer}
\label{def:reflection}
Let $(W,I)$ be a finite Coxeter system with longest element $w_0$,
and let $T=\{t_1,\dots,t_n\}$ ($n=\ell(w_0)$).  Then the 
total order  
$t_1 \prec t_2 \prec \dots \prec t_n$ on $T$ 
is a reflection order if and only
if there is a reduced expression $w_0 = s_{i_1} \dots s_{i_n}$
such that $t_j = s_{i_1} \dots s_{i_{j-1}} s_{i_j} s_{i_{j-1}} \dots s_{i_1}$,
for $1 \leq j \leq n$.
\end{definition}

\begin{remark}\label{reverse} \cite[Remark 2.4]{Dyer}
The reverse of a reflection order is a reflection order.
\end{remark}

\begin{proposition} \cite{Dyer} \label{ELDyer}
Fix a reflection order $\preceq$ on $T$.
Label
each edge $x \gtrdot y$ of the Bruhat order by the reflection
$x^{-1}y$.
Then this edge labeling together with $\preceq$ is an EL-labeling;
therefore the Bruhat order is EL-shellable.
\end{proposition}

In what follows, the notation $\hat{s}_k$ indicates the omission
of the factor $s_k$.

\begin{definition}\label{del_pair}\cite{Hersh}
Consider a Coxeter system $(W,I)$.
Define a {\it deletion pair} in an expression
$s_{i_1} \dots s_{i_d}$ to be a pair $s_{i_r}, s_{i_t}$ (where $r<t$) such that
the subexpression $s_{i_r} \dots s_{i_t}$ is not reduced but
$\hat{s}_{i_r} \dots s_{i_t}$ and $s_{i_r} \dots \hat{s}_{i_t}$
are each reduced.
\end{definition}

E.g.\ in type A the first $s_1$ and the last $s_2$ in
$s_1 s_2 s_1 s_2$ comprise a deletion pair.

\begin{lemma}\cite[Lemma 3.31]{Hersh}\label{TriciaLemma}
If $s_{i_r} \dots \hat{s}_{i_u} \dots s_{i_t}$ is reduced but
$s_{i_r} \dots s_{i_t}$ is not, then $s_{i_u}$ belongs to a deletion
pair within $s_{i_r} \dots s_{i_t}$.
\end{lemma}

\begin{proposition}\label{prop:reduced}
Consider $x\leq w$ in a Coxeter group $W$, and fix a reduced
expression $\mathbf w= (i_1,\dots,i_t)$ for $w$.  Let
$\mathbf x_+=\{j_1,\dots, j_k\}$.
For any $p \leq t$,  consider
the product $\gamma_1 \dots \gamma_t$, where
\begin{equation*}
\gamma_r=\begin{cases}
s_{i_r}, &\text{ if $r\in\mathbf x_+$ or $r\geq p$,}\\
1 & \text{otherwise.}
\end{cases}
\end{equation*}
Then $\gamma_1 \dots \gamma_t$ is reduced.
\end{proposition}
\begin{proof}
We will prove this by induction.  First consider $p=t$.  If
$i_t\in \mathbf x_+$ there is nothing to prove, since 
$s_{i_{j_1}} \dots s_{i_{j_k}}$ is reduced.
If $i_t \notin \mathbf x_+$ then assume that
$\gamma_1 \dots \gamma_t$ is not reduced.  This means that
$x s_{i_t} < x$, which 
contradicts the fact that $\mathbf x_+$ is the positive
subexpression for $x$.

Now by induction assume the proposition holds
for any $p$ between some $p'$ and $t$, where $p' \leq t$.
We want to prove it for $p:=p'-1$.
First suppose that $p'-1\in \mathbf x_+$.  In this
case the product $\gamma_1 \dots \gamma_t$ is the same for both
$p=p'$ and $p=p'-1$: in both cases, $\gamma_{p'-1}=s_{i_{p'-1}}$.
Therefore by induction it follows that $\gamma_1 \dots \gamma_t$ is reduced.

Now suppose that $p'-1 \notin \mathbf x_+$.
In this case the induction hypothesis
tells us only that $\gamma_1 \dots \gamma_{p'-2} \gamma_{p'} \dots \gamma_t$ 
is reduced;
we need to prove
that $\gamma_1 \dots \gamma_{p'-2} \gamma_{p'-1} \gamma_{p'} \dots \gamma_t=
      \gamma_1 \dots \gamma_{p'-2} s_{i_{p'-1}} s_{i_{p'}} \dots s_{i_t}$ 
is reduced.  Assume  it is not: then 
by Lemma \ref{TriciaLemma}, $s_{i_{p'-1}}$ belongs to a deletion pair
within 
$ \gamma_1 \dots \gamma_{p'-2} s_{i_{p'-1}} s_{i_{p'}} \dots s_{i_t}$. 
Note that $s_{i_{p'-1}} s_{i_{p'}} \dots s_{i_t}$
comprises a consecutive string of generators in a reduced expression and so
must be reduced.  Also note that by our argument in the first
paragraph,
$\gamma_1 \dots \gamma_{p'-2} s_{i_{p'-1}}$ must be reduced: otherwise
$\gamma_1 \dots \gamma_{p'-2} s_{i_{p'-1}} < \gamma_1 \dots \gamma_{p'-2}$,
which contradicts the fact that $\mathbf x_+$ is a positive
subexpression and does not contain $s_{i_{p'-1}}$.
But we've now shown that $s_{i_{p'-1}}$ cannot belong to a deletion pair
within 
$ \gamma_1 \dots \gamma_{p'-2} s_{i_{p'-1}} s_{i_{p'}} \dots s_{i_t}$,
a contradiction.
\end{proof}

\section{Morse matchings and the proof of contractibility}\label{CL}

In  this section we will construct a Morse matching on the face poset
of the closure of an arbitrary cell of $(G/P)_{\geq 0}$, such 
that matched edges are provably regular.  We will then 
use this to prove our main result: 
that the closure 
of each cell is contractible, and the boundary of each cell is 
homotopy equivalent to a sphere.

Recall the definition of the augmented face poset $\Q^J$ 
of $(G/P_J)_{\geq 0}$ from Section \ref{sec:regularity}.
Besides having a unique least element $\hat{0}$, $\Q^J$ also has a
unique greatest element: This  is $\hat{1}:=\mathcal P_{u_0, u_0,
w_0;>0}^J$, where $u_0$ and $w_0^J$ are the longest elements in $W_J$ and 
$W^J$, respectively.

The following  was proved in \cite{Wil}.

\begin{theorem}\cite{Wil}\label{LaurenTheorem}
$\Q^J$ is graded, thin, and EL-shellable. 
It follows that the face poset of 
$(G/P_J)_{\geq 0}$
is the face poset of 
a regular CW complex homeomorphic to a ball.
\end{theorem}

It will be useful for us to classify the cover relations in $Q^J$.
The following classification is analogous to the 
one used in \cite{Wil}, with the roles of $x$ and $w$ reversed.

\begin{lemma}\label{lem:cover}
The cover relations in $\Q^J$ fall into the following three categories.
\begin{description}
\item[Type 1] $\mathcal P_{x',v,w;>0}^J \lessdot \mathcal 
P_{x,u,w;>0}^J$ such that
     $x < x'$.  It follows that
     $xu^{-1} \lessdot x' v^{-1}$.
\item[Type 2] $\mathcal P_{x,v,w';>0}^J \lessdot \mathcal P_{x,u,w;>0}^J$ such that
     $w' \leq w$.  It follows that
     $w'v \lessdot wu$.
\item[Type 3]  $\hat{0} \lessdot \mathcal P_{x,u,w;>0}^J$ where $\mathcal P_{x,u,w;>0}^J$
    is a $0$-cell.  It follows  that $x=wu$.
\end{description}
\end{lemma}

\begin{remark}\label{interval}
If $Q$ is a poset, then
the {\it interval poset} $Int(Q)$ is defined to be the poset of
intervals $[x,y]$ of $Q$, ordered by containment.  
When $G/P_J$ is the complete flag variety, i.e. when
$J = \emptyset$, 
$\Q^J$ is simply the interval poset of the Bruhat order.
\end{remark}

\begin{theorem}\label{MorseMatch}
Choose any cell $\mathcal P^J_{x,u,w; >0}$
of $(G/P_J)_{\geq 0}$.
Then there is a 
Morse matching on the face poset of 
$\overline{\mathcal P^J_{x,u,w; >0}}$ with a single critical
cell of dimension $0$, which
restricts to a Morse matching on the face poset of 
the boundary 
$\bd({\mathcal P^J_{x,u,w; >0}})$ with one additional critical cell 
of top dimension.
Furthermore, all matched edges are good:
that is, if $\mathcal P_{x',u',w';>0}^J\lessdot \mathcal P_{x,u,w;>0}^J$
are matched, then $w'=w$ and there is a reduced expression 
$(i_1,\dots,i_m)$ of $w$
such that the positive subexpression 
$\mathbf{x'u'^{-1}_+}$ is equal to
$\mathbf{xu^{-1}_+} \cup \{k\}$ and moreover 
$\mathbf{xu^{-1}_+}$ contains $\{k+1,\dotsc, m\}$.
\end{theorem}

We will prove Theorem \ref{MorseMatch} in a series of steps.
Define $S_x(w):=\{\mathcal R_{v,w;>0} | x \leq v \leq w \}$,
and give this the poset structure inherited from $\Q^J$ (for $J=\emptyset$).
This poset is  isomorphic
to the (dual of the) Bruhat interval between $x$ and $w$.

\begin{proposition}\label{smallMorse}
$S_x(w)$ has a Morse matching $M_x(w)$ in 
which all matched edges 
are good.  If $x<w$, then $M_x(w)$  has no 
critical cells.  If $x=w$, $M_x(w)$  has one critical cell.
\end{proposition}

\begin{proof}
We will construct $M_x(w)$ by using Dyer's EL-labeling
of the Bruhat interval (Proposition \ref{ELDyer}) and
Chari's observation that one can go from a shelling
to a Morse matching (Proposition \ref{ChariMorse}).  To deduce
that matched edges are good, we will choose our reflection 
order carefully and use
Corollary \ref{Last}.

Fix a reduced expression $\mathbf w=({i_1},\dots, {i_m})$
for $w$, and choose
a reduced expression for 
$w_0$ which begins with
$\mathbf w^{-1}$.  
By Definition \ref{def:reflection},
this gives a reflection order.  Let 
$\prec$ be the reverse of this order; 
by Remark \ref{reverse}, $\prec$ is also a reflection order.

Label the edge  
$\mathcal R_{v',w;>0} \lessdot \mathcal R_{v,w;>0}$ (where $v'\gtrdot v$) 
in $S_x(w)$ with 
the reflection $\tau$ such that $v=v' \tau$.  
By Proposition \ref{ELDyer}, this gives an EL-labeling
of $S_x(w)$.  
If $S_x(w)$ has at least
two elements then by Corollary \ref{Bruhat-Matching}, there is
a Morse matching $M_x(w)$ on 
$S_x(w)$
with no critical cells.
Otherwise,
if $S_x(w)$ has  one element, i.e. if $x=w$,
then we take $M_x(w)$ to be the empty matching with  one 
critical cell.

We now need to show that all edges in $M_x(w)$ are good.  By Corollary
\ref{Last}, if $\tau$ labels
the edge  $\mathcal R_{v,w;>0} \gtrdot \mathcal R_{v',w;>0}$ 
(for $v' \gtrdot v$) and this edge
is in $M_x(w)$, then among all edge labels  going from $\mathcal R_{v,w;>0}$ to lower-dimensional
cells, $\tau$ is maximal in $\prec$.  So we need to analyze cover relations
corresponding to maximal labels.

Let $\mathbf v_+ = \{j_1,\dots,j_r\}$.
Let $k$ be maximal ($1 \leq k \leq m$) such that 
$k\notin \{j_1,\dots,j_r\}$.  We first claim that 
$\mathbf u = \{j_1,\dots,j_r\} \cup \{k\}$ is a reduced subexpression 
of $\mathbf w$, hence $\mathcal R_{u,w;>0} \lessdot \mathcal R_{v,w;>0}$,
and that $\mathbf u$ is positive.  Second,
we claim that 
the  label 
on the edge from $\mathcal R_{v,w;>0}$ to $\mathcal R_{u,w;>0}$
is maximal among all edge labels from $\mathcal R_{v,w;>0}$ down to a lower-dimensional cell.

Proposition \ref{prop:reduced} implies the first claim that
$\{j_1,\dots,j_r\} \cup \{k\}$ is a reduced subexpression 
of $\mathbf w$.  
Knowing that it is reduced, it is clear that it is positive.

To see that the second claim is true, note that 
by the choice of $k$, the  label 
on the edge  $\mathcal R_{v,w;>0} \gtrdot \mathcal R_{u,w;>0}$ is 
$u^{-1}v = s_{i_m} s_{i_{m-1}} \dots s_{i_k} \dots
s_{i_{m-1}} s_{i_m}$. 
Furthermore, in our reflection
order,
$s_{i_m} \succ s_{i_m} s_{i_{m-1}} s_{i_m} \succ \dots\succ 
s_{i_m} s_{i_{m-1}} \dots s_{i_k} \dots s_{i_{m-1}} s_{i_m} \succ \dots$.
Since $k$ is maximal such that $k \notin \{j_1,\dots,j_r\}$,
if we define $v^{''}=v s_{i_m} s_{i_{m-1}} \dots s_{i_{\ell}} \dots
s_{i_{m-1}} s_{i_m}$
for $\ell > k$, then $\ell \in \{j_1,\dots,j_r\}$
so an expression for $v^{''}$ is  
given by $\{j_1,\dots,j_r\} \setminus \{\ell\}$.  
In particular, $v^{''}<v$ and so $\mathcal R_{v,w;>0}$ does not cover
$\mathcal R_{v^{''},w;>0}$.  On the other hand, we know that 
$\mathcal R_{v,w;>0}$ covers $\mathcal R_{u,w;>0}$,  and that the label
on this edge is 
$s_{i_m} s_{i_{m-1}} \dots s_{i_{k}} \dots
s_{i_{m-1}} s_{i_m}$.   
By our choice
of reflection order,
this label is maximal among all edges from
$\mathcal R_{v,w}$ down to lower-dimensional cells.

Finally, by the choice of $k$, and since 
$\{j_1,\dots,j_r\}\cup \{k\} = \mathbf u$ is positive,
this cover relation is good.
Therefore every matched edge in $M_x(w)$ is good.
\end{proof}

\begin{remark}
Recently Brant Jones has constructed explicit 
matchings of the Hasse diagram of an interval in Bruhat order; 
he also proved that his matchings coincide with the matchings
$M_x(w)$ that we constructed in Proposition \ref{smallMorse} \cite{Jones}.
\end{remark}

\begin{remark}
If $x$ is the identity element in $W$, then the 
Morse matching constructed in Proposition 
\ref{smallMorse} will actually be a multiplication
matching by a Coxeter generator.  This is 
a so-called {\it special matching}, and is relevant to 
Kazhdan-Lusztig theory \cite{Brenti}.  
Anders Bjorner suggested using special matchings to 
construct acyclic matchings, and realized that 
one could use them to obtain an acyclic matching 
for the face poset of the entire space $(G/B)_{\geq 0}$
\cite{BjornerComment}.  We are grateful for his insights.
\end{remark}

We now turn to the proof of Theorem \ref{MorseMatch}

\begin{proof}
We partition the elements of the face poset  
of the 
closure of 
$\mathcal P_{x,u,w; >0}^J$ into subsets 
$S^{J}_{xu^{-1}}(y)=
\{\mathcal P_{x',u',y;>0}^J \ \vert \ xu^{-1} \leq x'u'^{-1} \leq y\}$, 
for each $y\in W^J$ such that
$xu^{-1} \leq y \leq w$.
By Lemma \ref{lem:cover}, the restriction of the face poset $\Q^J$ to 
${S}^J_{xu^{-1}}(y)$
is isomorphic to the (dual of the) Bruhat interval between $xu^{-1}$ and $y$,
so ${S}^J_{xu^{-1}}(y)$ and $S_{xu^{-1}}(y)$ are isomorphic as posets:
we simply identify $\mathcal P_{a,b,y;>0}^J$ with $\mathcal R_{ab^{-1},y;>0}$.

We can now apply Proposition \ref{smallMorse}, which gives
us a Morse matching $M^J_{xu^{-1}}(y)$
on $S^J_{xu^{-1}}(y)$ such that all matched 
edges are good.  This matching has either zero or one critical cell,
based on whether $xu^{-1} < y$ or $xu^{-1}=y$.

We now define $$\mathcal M^J_{x,u,w} = \cup_{y\in W^J, xu^{-1} \leq y \leq w}
                   M^J_{xu^{-1}}(y).$$
Since each $M^J_{xu^{-1}}(y)$ is a matching,
and any two matched elements $\mathcal P_{a,b,y;>0}^J$ and $\mathcal P_{a',b',y;>0}^J$ 
in $\mathcal M^J_{x,u,w}$
have the 
same third factor $y$,  $\mathcal M^J_{x,u,w}$ is also a matching.

Let us assume for the sake of contradiction that 
there is a cycle in $G(\mathcal M^J_{x,u,w})$.
Since each $G(M^J_{xu^{-1}}(y))$
is acyclic, there must be some edges in the cycle which pass between two different
$S^J_{xu^{-1}}(y)$'s.  Each such edge must be directed from the 
higher-dimensional cell $\mathcal P_{a,b,y;>0}^J$ to the lower-dimensional cell
$\mathcal P_{a',b',y';>0}^J$ for $y \neq y'$, so 
by Lemma \ref{lem:cover}, $y'<y$.   So if we traverse the cycle and 
look at the sequence of poset elements $\mathcal P_{*,*,y;>0}^J$ that
we encounter, the third factor
can only decrease.  Therefore it is impossible to return to the 
element of the cycle at which we started, which is a contradiction.

As $y$ varies over elements of $W^J$ between $xu^{-1}$ and $w$, 
we have that $M^J_{xu^{-1}}(y)$ has no critical cells for $xu^{-1}<y$
and it has one critical cell $\mathcal P_{x,u,xu^{-1};>0}^J$
for $xu^{-1}=y$.  Therefore 
$\mathcal M^J_{x,u,w}$ has a unique critical cell,
the $0$-dimensional cell $\mathcal P_{x,u,xu^{-1};>0}^J$.

Since the face poset of $\overline{\mathcal P^J_{x,u,w;>0}}$
has a unique  cell of top dimension $\ell(w)-\ell(xu^{-1})$
which is matched in $M^J_{x,u,w}$, when we restrict $M^J_{x,u,w}$
to 
the boundary $\bd({\mathcal P^J_{x,u,w;>0}})$, we will get a Morse matching
with one additional critical
cell of top dimension $\ell(w)-\ell(xu^{-1})-1$.
This completes the proof of the theorem.
\end{proof}

\begin{corollary}\label{cor:matching}
Choose any cell $\mathcal P^J_{x,u,w;>0}$ of 
$(G/P_J)_{\geq 0}$.  Then there is a Morse matching on the face poset 
 of 
$\overline{\mathcal P^J_{x,u,w;>0}}$
with a single critical cell 
of dimension $0$ in which all matched edges are regular;
it restricts to a Morse matching on the face poset of the boundary
$\bd({\mathcal P^J_{x,u,w;>0}})$ 
with one additional critical cell of top 
dimension.
\end{corollary}

\begin{proof}
This follows from Theorem \ref{regularity}
and Theorem \ref{MorseMatch}.
\end{proof}

We now prove our main result.

\begin{theorem}
The closure of each cell of $(G/P)_{\geq 0}$ is contractible,
and the boundary of each cell of $(G/P)_{\geq 0}$ is homotopy
equivalent to a sphere.
\end{theorem}

\begin{proof}
Choose an arbitrary cell of 
$(G/P)_{\geq 0}$  and let $K$ be its closure.
Note that $K$ is a CW complex with the subcomplex property because 
Theorems \ref{RWCW} and \ref{RTheorem} imply that 
$(G/P)_{\geq 0}$ is.  
Let $Q$ be the face poset of $K$.
By Corollary \ref{cor:matching}, $Q$ has a 
Morse matching with a unique critical 
cell of dimension $0$, in which all matched edges are regular.
Therefore Theorem \ref{MainMorseTheorem} implies $K$ is contractible.

Now let $K'$ be the boundary of an arbitrary  
cell and let $Q'$ 
be its face poset.
By Corollary \ref{cor:matching}, $Q'$ has a 
Morse matching with two critical cells, one of dimension $0$
and one of top dimension, say $p$,
in which all matched edges are regular.
Therefore Theorem \ref{MainMorseTheorem} implies that $K'$ is homotopy
equivalent to a CW complex with one $0$-dimensional cell $\sigma$ and 
one $p$-dimensional cell  whose boundary is glued to $\sigma$.
This is precisely a $p$-sphere.
\end{proof}

\begin{remark}
Since a Morse function actually gives rise
to a 
concrete {\it collapsing} \cite{RF} of a CW complex, in fact we have shown 
that the
closure of a cell is {\it collapsible.}
\end{remark}

\begin{remark}
One can give a  simpler proof that the closure of each cell
in the totally non-negative part of the type A Grassmannian 
$(Gr_{kn})_{\geq 0}$ is contractible.  In that case, one can prove
directly that whenever a cell
$\sigma$ has codimension $1$ in the closure of $\tau$, then 
$\sigma$ is a regular face of $\tau$.  This follows from the 
technology of \cite{Postnikov}: in particular, 
Theorem 18.3, Lemma 18.9, and Corollary 18.10.  Then by 
Theorem \ref{LaurenTheorem}, the poset of cells of 
$(Gr_{kn})_{\geq 0}$ is a CW poset with an EL-labeling (hence a shelling), 
so by 
Proposition \ref{ChariMorse}, we have the requisite Morse matching.
\end{remark}

Using Corollary \ref{algo}, we see that 
there is a more concrete way to describe the matchings 
$M_x(w)$ of $S_x(w)$.

\begin{remark}
Fix a reduced expression ${\mathbf w}=({i_1},\dots, {i_m})$ for $w$.
Start with
the maximal element ${\mathcal R_{x,w;>0}}$.
Let $k$ be maximal such that $1\leq k\leq m$
and $k \notin \mathbf x_+$.
Then $\mathbf x_+ \cup \{k\}$ is the positive 
subexpression for an element $v>x$.  
We match $\mathcal R_{x,w;>0}$ to $\mathcal R_{v,w;>0}$.
Now apply the same procedure 
to every element of dimension 
$\dim \mathcal R_{x,w;>0} - 1$  which has not been matched
(the order in which we consider these elements does not matter).
Continue in this fashion, from higher to lower-dimensional cells.
\end{remark}



\begin{thebibliography}{2}

\bibitem{BZ} A.~Berenstein, A.~Zelevinsky, Tensor product
multiplicities, canonical bases and totally positive varieties,
Invent. math., \textbf{143} (2001), no.~1, 77--128.

\bibitem{B1} A. Bjorner, Shellable and Cohen-Macaulay partially
ordered sets, Trans. Amer. Math. Soc. 260 (1980), 159--183.

\bibitem{B2} A. Bjorner, Posets, regular CW complexes and Bruhat
order, Europ. J. Combin. 5 (1984), 7--16.


\bibitem{RedBook} A. Bjorner, M. Law Vergnas, B. Sturmfels, N. White,
G. Ziegler, Oriented Matroids, Cambridge University Press, 1993.

\bibitem{BW1} A. Bjorner, M. Wachs, Bruhat order of Coxeter groups
and shellability, Adv. Math. 43 (1982) 87--100.

\bibitem{BW2} A. Bjorner, M. Wachs, On lexicographically shellable
posets, Trans. Amer. Math. Soc. 277 (1983), 323--341.

\bibitem{BW3} A. Bjorner, M. Wachs, Shellable nonpure complexes and posets I,
Trans. Amer. Math. Soc. 348 (1996), no. 4, 1299--1327.

\bibitem{BW4} A. Bjorner, M. Wachs, Shellable nonpure complexes and posets II,
Trans. Amer. Math. Soc. 349 (1997), no. 10, 3945--3975.

\bibitem{BjornerComment} A. Bjorner, personal communication to L. Williams,
Mittag-Leffler Institute, May 2005.


\bibitem{Brenti} F. Brenti, F. Caselli, M. Marietti, Special matchings
and Kazhdan-Lusztig polynomials, Adv. Math. 202 (2006), 555--601.

\bibitem{Chari}  M. Chari, On discrete Morse functions and combinatorial
decompositions, Discrete Math 217 (2000), 101 -- 113.


\bibitem{DK} G. Danaraj, V. Klee, Shellings of spheres and polytopes,
Duke Math. J. 41 (1974), 443--451.

\bibitem{Deodhar} V. Deodhar, A combinatorial setting for questions in 
Kazhdan-Lusztig theory, Geometriae Dedicata 36 (1990), 95--119.

\bibitem{Dyer} M. Dyer, Compos. Math. 89 (1993), 91--115.


\bibitem{RF} R. Forman, Morse theory for cell complexes, Adv. Math. 134
(1998), no. 1, 90--145.

\bibitem{Forman2} R. Forman, A user's guide to discrete Morse theory, 
Sem. Lothar. Combin. 48 (2002), Art. B48c, 35pp. (electronic).

\bibitem{Fulton} W. Fulton, Introduction to toric varieties,
Annals of Mathematics Studies, 131.  The William H. Roever Lectures
in Geometry, Princeton University Press, Princeton, NJ, 1993.

\bibitem{GoodearlYakimov} K. R. Goodearl, M. Yakimov, Poisson structures
on affine spaces and flag varieties II. General case, arXiv:math.QA/0509075
(to appear in Trans. AMS)

\bibitem{Hersh} P. Hersh, Regular cell complexes in total 
positivity, arXiv:0711.1348.

\bibitem{Humphreys} J. Humphreys, Reflection groups and Coxeter
groups, Cambridge University Press, 1990.

\bibitem{Jones} B. Jones, An explicit derivation of the Mobius
function for Bruhat order, arXiv:0904.4472.

\bibitem{KL} D. Kazhdan, G. Lusztig, Representations of Coxeter
groups and Hecke algebras, Invent. Math., 53 (1979),
165--184.

\bibitem{Kozlov} D. Kozlov, Collapsibility of 
$\Delta(\Pi_n) \/S_n$ and some related CW complexes, Proc. Amer. Math.
Soc. 128 (2000), no. 8, 2253--2259.


\bibitem{Lus:CanonBasis}
G.~Lusztig, Canonical bases arising from quantized enveloping
algebras, J. Amer. Math. Soc. 3 (1990), 447--498.

\bibitem{Lus:Quantum}
G.~Lusztig, Introduction to Quantum Groups, Progr. in Math. 110, Birkhauser
Boston, 1993, 341 pages.

\bibitem{Lusztig3}
G. Lusztig, Total positivity in reductive groups, in:
Lie theory and geometry: in honor of Bertram Kostant,
Progress in Mathematics 123, Birkhauser, 1994.

\bibitem{Lusztig1} G. Lusztig, Introduction to total positivity, in
Positivity in Lie theory: open problems, ed. J. Hilgert, J.D. Lawson,
K.H. Neeb, E.B. Vinberg, de Gruyter Berlin, 1998, 133--145.

\bibitem{Lusztig2} G. Lusztig, Total positivity in partial flag manifolds,
Representation Theory, 2 (1998) 70--78.

\bibitem{LusztigQuestion} G. Lusztig, personal communication with
K. Rietsch, Oberwolfach, 1996.

\bibitem{MarRie:ansatz}
R. Marsh, K. Rietsch, Parametrizations of flag varieties, Representation
Theory, 8 (2004).


\bibitem{Postnikov} A. Postnikov,
Total positivity, Grassmannians, and networks,
 arXiv:math.CO/0609764.

\bibitem{PSW} A. Postnikov, D. Speyer, L. Williams,
Matching polytopes, toric geometry,
and the non-negative part of the Grassmannian,
to appear in J. Alg. Combin.


\bibitem{Rietsch1} K. Rietsch, Total positivity and real flag varieties,
Ph.D. Dissertation, MIT, 1998.

\bibitem{Rietsch2} K. Rietsch, Closure relations for totally non-negative
cells in $G/P$, 
 Math. Res. Let. \textbf{13} (2006), 775--786.

\bibitem{Rie:MSgen}
K.~Rietsch, A mirror symmetric construction for $q{H}_{T}^*({G}/{P})$,
   Adv. Math. \textbf{217} (2008), 2401--2442.


\bibitem{RW} K. Rietsch, L. Williams, The totally non-negative part 
of $G/P$ is a CW complex,  Transformation Groups \textbf{13}, Special volume in honor of B. Kostant's 80th 
birthday, (2008), 839--953.

\bibitem{Sottile} F. Sottile, Toric ideals, real toric varieties,
and the moment map.  Topics in algebraic geometry and geometric modeling,
225--240, Contemp. Math., 334, Amer. Math. Soc., Providence, RI, 2003.

\bibitem{Steinberg} R. Steinberg, Endomorphisms of linear algebraic
groups, Memoirs of the Amer. Math. Soc., 80, Amer. Math. Soc.,
Providence, RI, 1968.

\bibitem{Wachs} M. Wachs, Poset topology: tools and applications.
Geometric combinatorics, 497--615, IAS/ Park City Math. Ser., 13, 
Amer. Math. Soc., Providence, RI, 2007.


\bibitem{Wil} L. Williams,
Shelling totally non-negative flag varieties,
 J. Reine Angew. Math. \textbf{609}, 2007.


\end{thebibliography}
\end{document}